\RequirePackage{fix-cm}
\documentclass[smallextended]{svjour3}       % onecolumn (second format)
\smartqed  % flush right qed marks, e.g. at end of proof

\usepackage{amsmath}
\usepackage{amssymb}   
\usepackage{graphicx}

\usepackage{comment}

\usepackage{hyperref}
\hypersetup{colorlinks,linkcolor={blue},citecolor={blue},urlcolor={red}}

\newtheorem{thm}{Theorem}[section] 
 
\newtheorem{cor}[thm]{Corollary} 
\newtheorem{prop}[thm]{Proposition} 
 
\newtheorem{defn}[thm]{Definition} 
\newtheorem{rem}[thm]{Remark}

\usepackage[labelsep=endash]{caption}

 \numberwithin{equation}{section}
\numberwithin{thm}{section}

\def\p{\partial}
\def \ds{\displaystyle}
\def \vs{\vspace*{0.1cm}}

\usepackage{latexsym}
\begin{document}

\title{A mean curvature type flow with capillary boundary in a unit ball%\thanks{Grants or other notes
}


\author{Guofang Wang        \and
        Liangjun Weng %etc.
}


\institute{Guofang Wang \at
             Mathematisches Institut, Albert-Ludwigs-Universit\"{a}t Freiburg, Freiburg im Breisgau, 79104, Germany \\
            %  Tel.: +123-45-678910\\
           %   Fax: +123-45-678910\\
              \email{guofang.wang@math.uni-freiburg.de}           %  \\
           \and
           Liangjun Weng \at
           School of Mathematical Sciences, University of Science and Technology of China, Hefei, 230026, P. R. China\\
               Mathematisches Institut, Albert-Ludwigs-Universit\"{a}t Freiburg, Freiburg im Breisgau, 79104, Germany
             \\  \email{ljweng08@mail.ustc.edu.cn; liangjun.weng@math.uni-freiburg.de} 
}

\date{Received: date / Accepted: date}

\maketitle

\begin{abstract}
In this paper, we study a mean curvature type flow with  capillary boundary in the unit ball. Our flow preserves the volume of the bounded domain enclosed by the hypersurface, and monotonically decreases an energy functional $E$. We show that it has the longtime existence and subconverges to spherical caps. As an application, we solve an isoperimetric problem for hypersurfaces with  capillary boundary.
\keywords{Mean curvature flow \and capillary boundary \and conformal Killing vector field \and a priori estimate \and isoperimetric  problem}
\subclass{Primary: 53C44 Secondary: 35K93}
\end{abstract}

\section{Introduction}
\label{intro}
In this paper, we are interested in a  mean curvature type flow in the unit ball $\mathbb{B}^{n+1}\subset \mathbb{R}^{n+1}$ with capillary boundary. 
Roughly speaking, given a Riemannian manifold $N^{n+1}$ with a smooth boundary $\partial N$, 
a hypersurface with capillary boundary in $N$ is an immersed hypersurface which intersects $\partial N$ at a constant contact angle $\theta\in (0,\pi)$.

For closed hypersurfaces, the mean curvature flow plays an important role in geometric analysis and has been extensively studied. One of 
classical results proved by  Huisken \cite{Hui1984} states that it contracts a closed convex  hypersurface into a round point.
Mean curvature type flows with a constraint play an important role in the study of isoperimetric problems.
The following curve-shortening (and area-preserving) flow was studied  by Gage \cite{Gage1986}. 
Let $\gamma:\mathbb{S}^1\times [0,T)\to \mathbb{R}^2$ satisfy \begin{center}
	$\partial_t\gamma=(\kappa-\frac{2\pi}{L})\nu,$
\end{center}
where $\kappa$ is the geodesic  curvature of $\gamma$, $L$ is the length of the curve at scale $t$, and $\nu$ is the outward unit normal vector of curve $\gamma(\cdot,t)$.
In a higher dimensional  Euclidean space, Huisken introduced a non-local type mean curvature flow in \cite{Hui1987}: Given a closed, connected 
hypersurface $M$, consider a family of embeddings $x:M\times[0,T)\to\mathbb{R}^{n+1}$ satisfies 		\begin{center}
	$	\partial_tx= (c(t)-H)\nu,$
\end{center}
where $c(t):=\frac{\int_{M_t} H d\mu}{|M_t|}$ is the average of the mean curvature $H$ of $M_t:=x(M,t)$
and $\nu$ is the unit outward normal vector field of $M_t$.  Huisken proved that such a  volume preserving flow converges to a round sphere 
if the initial hypersurface is uniformly convex.  %Both results imply a solution of a corresponding isoperimetric problem.
There has been a lot of work on  such geometric flows. Here
we  just  mention  further \cite{AF} for studying such kind of 
flow in the case where the ambient space is Riemannian manifold and  \cite{Mc} for the extension to a general mixed volume preserving mean
curvature flow.  As one of applications, such a volume (or area)-preserving flow could be used to prove optimal geometric inequalities.  In order to  establish optimal geometric inequalities, 
there is another type mean curvature flow, which is first introduced by Guan and Li \cite{GL} inspired by the Minkowski formulas.
A flow $x:M\times [0,T)\to\mathbb{M}^{n+1}_k$ satisfies
\begin{center}$
	\partial_tx=(n\phi'(\rho)-Hu)\nu,$
\end{center}
where $u$ is the support function of hypersurface $x(M,\cdot)$ and $\mathbb{M}^{n+1}_k$ is the space form with constant sectional curvature
$k$ and metric $ds^2:=d\rho^2+\phi^2(\rho)g_{_{\mathbb{S}^{n}}}$. This flow is also volume preserving and 
area decreasing by the Minkowski formulas. They obtained the longtime existence of this flow and proved that it smoothly converges to a round
sphere if the initial hypersurface is star-shaped. 
As a result, this yields a flow proof of classical Alexandrov-Fenchel inequalities of quermassintegrals in convex geometry. 
Recently, they obtained that a similar phenomenon also holds for the general warped produced manifold in \cite{GuanLiWang} 
jointed with Wang.  For the methods which use a fully nonlinear flow to establish geometric inequalities, we refer also to \cite{GuanWang}.
Last but not least, we recommend the literature \cite{A}, \cite{AW}, \cite{AW2}, \cite{CM}, \cite{SX} 
and references therein for  extensions to general anisotropic and  fully nonlinear curvature flows in various  ambient spaces.

 There has been a great interest in the investigation of hypersurfaces with non-empty boundaries in the last thirty years. 
For instance, Stahl \cite{St} considered the mean curvature flow with free boundary in the  Euclidean space, and he showed 
that the solution either has the longtime existence or the curvature and its derivatives blow up at the maximal time.  Later,
Marquardt \cite{M1} considered the inverse mean curvature flow for  hypersurfaces with boundary perpendicular to a convex cone, 
and proved that it has the long time existence and  converges to a piece of round sphere, if the initial hypersurface is star-shaped 
and strictly mean convex. Recently  Lambert-Scheuer \cite{LS2} studied the same flow as \cite{M1} but with the supporting hypersurface
being a sphere instead of a cone. They proved that a convex hypersurface  which is perpendicular to a sphere 
along the boundary converges to a flat disk in certain sense. 
As  a nice geometric application of this flow they proved in \cite{LS1} a Willmore type inequality. 
We also would like to mention the recent articles \cite{WX2} and \cite{SWX} for a  mean curvature type flow and a fully nonlinear inverse curvature type flow respectively in the unit ball with free boundary, where new geometric inequalities were proved as applications.
 For the study of a nonparametric mean curvature flow
with free or capillary type boundaries, we refer to  \cite{AWu}, \cite{GMWW}, \cite{LW1} and \cite{Hui1989}.

Those results motivate us to consider the following mean curvature type flow for hypersurfaces  with capillary boundary. To be more precise, let $\Sigma_0$ be a properly embedded  compact smooth hypersurface in $\overline{\mathbb{B}}^{n+1} (n\geq 2)$ with capillary boundary $\partial \Sigma_0\subset \mathbb{S}^n:=\partial \mathbb{B}^{n+1}$, which is given by $x_0:M \to \overline{\mathbb{B}}^{n+1}$ and $M$ is a compact manifold with smooth boundary $\partial M$. 
In other words, $\mbox{int}(\Sigma_0)=x_0(\mbox{int}(M))$, and 
$\partial\Sigma_0=x_0(\partial M)$ intersects $\partial \mathbb{B}^{n+1}$ at a constant contact angle $\theta\in(0,\pi)$. Consider a family of embeddings
$x:M\times [0,T) \to \overline{\mathbb{B}}^{n+1}$ with $x(\partial M,\cdot )\subset \partial\overline{\mathbb{B}}^{n+1}$   such that

\begin{equation}\label{MCF with angle}
\begin{array}{rcll}
(\partial_t x )^\perp&=&f\nu \quad &
\hbox{ in }M\times[0,T),\\
\langle \nu,\overline{N}\circ x \rangle &=& -\cos \theta 
\quad & \hbox{ on }\partial M\times [0,T),\\
x(\cdot,0) & =& x_0(\cdot) \quad & \text{ on }   {M}, 
\end{array}
\end{equation}
where $$f:=n\langle x,a\rangle+n\cos\theta \langle \nu,a\rangle -H\langle X_a,\nu\rangle, \qquad \hbox{ for } a\in {\mathbb S}^n ,$$  $\nu$ and $
H$ are the unit normal vector field and the  mean curvature of hypersurface $x(\cdot,t)$ resp.,
$\overline{N}$ is the unit outward normal vector field of $\mathbb{S}^n$ , the contact angle $\theta\in  (0, \pi)$  is
a constant and the vector field $X_a$ will be defined and discussed in the next paragraph.
  Here, for a vector field $\xi$ along a hypersurface $x$, we define its normal part by $\xi^\perp:= \langle \xi, \nu\rangle \nu.$
The choice of $f$ is motivated by new Minkowski formulas   proved in  \cite{WX1}.
%which is a very  natural from the viewpoint of the  Minkowski formula for capillary hypersurfaces  as in  \cite{GL}, \cite{GuanLiWang} mentioned above
%or \cite{WX2}.  
If $\theta=\frac{\pi}{2}$, it corresponds to the free boundary problem of parabolic setting, which was studied by Wang and Xia in \cite{WX2}.

Before we state our main results, we clarify the notation $X_a$ used above. In this paper  $X_a$ is a vector field defined by 
\begin{equation}\label{a_1}
X_a:=\langle x,a\rangle x-\frac{1}{2}\big(|x|^2+1\big)a,
\end{equation}
where $a$ is a fixed unit vector in $\mathbb{R}^{n+1}$. One can easily to check 
that $X_a$ is a conformal Killing vector field in $\mathbb{B}^{n+1}$. In fact,  $X_a$ is exactly the pull back of the position vector field under  
a conformal transformation from the unit ball to the half Euclidean space.  See Section 3.2 for the precise discussion. 
We say that a properly embedded hypersurface $\Sigma$ in $\mathbb{R}^{n+1}$ is star-shaped with respect to $a$ if $\Sigma$ intersects each integral curve of $X_a$ only once. For simplicity
in this paper we define a hypersurface of star-shaped by a stronger condition that 
%Moreover,  we called a hypersurface $\Sigma:=x(M)$ with capillary boundary {\it  strictly star-shaped} if
%it satisfies
$$
\langle X_a,\nu\rangle >0
$$
holds everywhere in $M$.  % It is clear that a hypersurface of strictly star-shaped is of star-shaped. In this paper we  study hypersurfaces of strictly star-shaped and   for the simplicwe call 
Our main theorem is the following 
\begin{thm}\label{Main thm}If the initial hypersurface is a  star-shaped hypersurface  with capillary boundary in the unit ball and the
	contact angle $ \theta$ satisfies $| \cos \theta | < \frac{3n+1}{5n-1}$, then  flow  \eqref{MCF with angle} exists globally 
	with uniform $C^{\infty}$-estimates. Moreover,  $x(\cdot,t)$  subsequently  converges  to a spherical cap  in the $C^{\infty}$ topology as $t\to \infty$, whose enclosed domain has the same volume as the domain enclosed by $\Sigma_0$.
\end{thm}
If $\theta=\frac{\pi}{2}$, i.e., the free boundary case, this theorem was proved recently by Wang and Xia in \cite{WX2}, where they also proved the global  convergence.
The free boundary case usually corresponds to a homogeneous or linear Neumann boundary value conditions,
see \cite{Hui1989}, \cite{LS1}, \cite{M1}, \cite{SWX}, \cite{St} and \cite{WX2} for example. 
However, the capillary boundary case  in general  relates to a nonlinear type Neumann boundary value condition, which is more complicated and 
technically more  difficult to handle from the analytic  viewpoint. This difficulty usually prevents us to obtain estimates for a full range of 
$\theta \in (0, \pi)$. For instance, Guan  obtained the gradient estimate (depending on the time $T$) of solution in \cite{Guan1998} for a nonparametric curvature flow with capillary boundary for angle satisfying $|\cos\theta|<\frac{\sqrt{3}}{2}$.
Recently, the authors \cite{GMWW} obtained the uniformly gradient estimate (independent of time $T$) for the nonparametric mean curvature flow with capillary boundary for $\theta$ 
in a small 
neighborhood of $\frac \pi 2$.
In this paper we obtain for our flow \eqref{MCF with angle} a better range $|\cos\theta|<\frac{3n+1}{5n-1}$.
The reason why we can have a bigger  range of the contact angle is due to an observation that  
equation \eqref{salar equation of tilde} has  a good term when we carry out the gradient estimate.  
See the proof of Proposition \ref{gradient estimate uniformly}. Also due to this difficulty, 
we can only prove the subsequence convergence of this flow and are  not  able to show that the limits are the same spherical cap at the moment. We will consider this problem in the near future.
Nevertheless, the limits  have the same radius and hence we can provide a flow proof for 
the isoperimetric problem for hypersurfaces with capillary boundary in the unit ball %and Theorem \ref{Main thm}, we have the following nice corollary.
\begin{cor}
	Among   star-shaped capillary boundary hypersurfaces with a volume constraint  the spherical caps given in Remark \ref{static model} are the only minimizers of  the energy  functional $E$ defined in \eqref{total energy} below, provided that 
	 the contact angle $\theta$ satisfies
	$|\cos\theta|<\frac{3n+1}{5n-1}$,.
	
	% For the definition of $E$, see \eqref{E} below.
\end{cor}
%This family of spherical caps is given in Remark \ref{static model}.

The Corollary follows from Theorem \ref{Main thm} and the crucial properties that the flow preserves the enclosed  volume and decreases the energy functional $E$, which are proved  in Subsection 2.3 by the new Minkowski formulas
established in \cite{WX1}.

\

\noindent\textit{This article is organized as follows.} In Section 2, we give some preliminaries about hypersurfaces with capillary boundary 
and our mean curvature type flow.
In Section 3, we convert the flow to a scalar equation on semi-sphere with the help of a conformal transformation. 
In the last Section, we establish a priori estimates and prove the main theorem.

\section{Preliminaries} 
 
 In this Section we provide basic facts of capillary hypersurfaces and prove the crucial facts of our flow in Proposition \ref{energy de}  by using the new Minkowski formulas obtained in \cite{WX1}.
 For convenience of the reader we provide complete proofs.
 For more  information about capillary hypersurfaces we  refer to the wonderful exposition book \cite{Finn1986}.
 
 \subsection{Integral identities}
% To begin with, we deduce some properties of hypersurfaces $
In this paper we consider hypersurfaces  $\Sigma\subset\overline{\mathbb{B}}^{n+1}$ with capillary boundary $\partial\Sigma$ on $\partial \overline{\mathbb{B}}^{n+1}$  which will be precisely defined below. 
Since we will use a flow to study such hypersurfaces, it will convenience to use the parametrization:
Let $x:M\to \overline{\mathbb{B}}^{n+1}$ be an isometric immersion of an orientable $n$-dimensional compact manifold $M$ with smooth boundary $\partial M$ such that  $\Sigma:=x(M)$ and
 $\partial\Sigma:=x(\partial M)$. %We use the Einstein summation convention in this section, which means that the repeated arabic indices $i,j,k$ should be summed from $1$ to $n$, if not stated.
 However, we will identity $M$ with $\Sigma$ and $\p M$ with $\p\Sigma$, if there is no confusion.

 Let $\overline N$ be the unit outward normal $\overline{N}$ of the unit sphere $\partial \mathbb{B}^{n+1}$.
 Let $\Sigma\subset \overline{\mathbb{B}}^{n+1}$ be 
 a smooth oriented hypersurface  with boundary $\partial \Sigma$ satisfying 
 $\mbox{int}(\Sigma)\subset \mathbb{B}^{n+1}$ and $\partial \Sigma\subset \partial \mathbb{B}^{n+1}$.
 $\Sigma$ divides the unit ball into  two parts. We denote one part by $\Omega$ and define $\nu$ the unit outward normal vector field of $\Sigma$ w.r.t. 
 $\Omega$. 
 Let $\mu$ be  the unit outward conormal vector field along $\partial\Sigma$  and 
 $\overline{\nu}$ be the unit normal to $\partial\Sigma$ in $\partial\mathbb{B}^{n+1}$ such that $\{\nu,\mu\}$ and $\{\overline{\nu},\overline{N}\}$ 
 have the same orientation in the normal bundle of $\partial\Sigma\subset\overline{\mathbb{B}}^{n+1}$. See Figure 1.

 We call the angle between  $-\nu $ and $\overline N$
  {\it contact angle} and denote it by $\theta$. It follows
 \begin{equation*}
 \begin{array}{rcl}
 \overline{N}&=&\sin\theta \mu-\cos\theta \nu,
 \\
 \overline{\nu}&=&\cos\theta \mu+\sin\theta \nu.
 \end{array}
 \end{equation*}
 or equivalently
 \begin{equation*}
 \begin{array}{rcl}
 \mu&=&\sin\theta \overline{N}+\cos\theta \overline{\nu}
 \\
 \nu&=&-\cos\theta \overline{N}+\sin\theta \overline{\nu}.
 \end{array}
 \end{equation*}
     \begin{figure} \vspace{20mm}\includegraphics[width=0.6\linewidth]{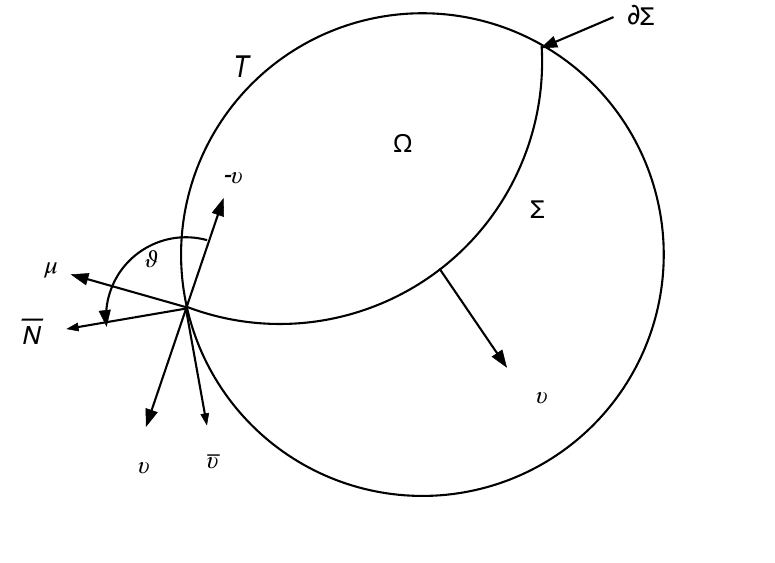} \label{fig1}  \caption{ $\Sigma=x(M)$  and $ \partial\Sigma=x(\partial M)$} \end{figure}

 %we call $\partial\Sigma$ is a capillary boundary, if the unit outward conormal vector field $\mu$ along $\partial\Sigma$ and the unit normal $\overline{N}$ along $\partial \mathbb{B}^{n+1}$ intersect at a constant angle $\theta\in (0,\pi)$. That is, 		\begin{center}
 %$	\langle \mu,\overline{N}\rangle=\sin\theta,$
 %\end{center} holds on $\partial \Sigma$. 
% In particular, 
 \begin{defn}\label{capillary bdry}
 	Given a smooth oriented hypersurface $\Sigma\subset \overline{\mathbb{B}}^{n+1}$ with $\mbox{int}(\Sigma)\subset \mathbb{B}^{n+1}$ and 
 	$\partial \Sigma\subset \partial \mathbb{B}^{n+1}$, we call that  $\partial\Sigma$ is a capillary boundary, if the contact angle $\theta\in(0,\pi)$ is constant along $\partial \Sigma$.
 	%if the unit outward conormal vector field $\mu$ along $\partial\Sigma$ and the unit normal $\overline{N}$ along $\partial \mathbb{B}^{n+1}$ intersect at a constant angle $\theta\in (0,\pi)$. That is, 		\begin{center}
 	Namely,
 	$$\langle \mu,\overline{N}\rangle=\sin\theta$$
 	is constant on $\partial \Sigma$. 
 	In particular, if $\theta=\frac{\pi}{2}$,  i.e.,  $\Sigma$ meets $\partial \mathbb{B}^{n+1}$ orthogonally, we call that $\partial\Sigma$ is a free boundary.	
 \end{defn}

 We denote $D$ and $\nabla$ derivatives on $(\mathbb{B}^{n+1},\delta_{\mathbb{B}^{n+1}})$
 and $(M,g)$ resp., where $\delta_{\mathbb{B}^{n+1}}$ is the standard Euclidean metric and $g$ is the induced metric on $M$.

 Recall that $X_a$ is the conformal vector field defined  by \eqref{a_1}.
 Decompose $X_a$ into  $ { X_a:=X_a^T+\langle X_a,\nu\rangle\nu} $, where $X_a^T$ is the tangential projection of $X_a$ on $\Sigma$. 
 It is clear to see that $X_a:=\langle x,a\rangle x-a$  on $\p\Sigma$ and $\overline{N}=x$ on 
 $\partial \mathbb{B}^{n+1}$, which follows that   
 \begin{equation}\label{tang killing metric}
 \begin{split}
 \langle X_a^T,\mu\rangle & =\langle X_a,\mu\rangle = \langle X_a, \sin\theta N+\cos\theta \overline{\nu}\rangle
 \\& = \cos\theta\langle X_a,\overline{\nu}\rangle= \cos\theta \langle \langle x,a\rangle x-a,\overline{\nu}\rangle= -\cos\theta \langle a,\overline{\nu}\rangle.
 \end{split}
 \end{equation}
 Let $h$ be the second fundamental form of the hypersurface $\Sigma$ given by $h(X,Y):=\langle D_X\nu,Y\rangle$ for any $X,Y\in T\Sigma$ with $\Sigma:=x(M)$ and $H$ is the mean curvature of $\Sigma$.
 Note that  \begin{align}\label{vanish 2nd F.F. of e,mu}
  h(e,\mu)=0
 \end{align} for any $e\in T(\partial \Sigma)$ and \begin{align}
  D_{\mu}\nu=h(\mu,\mu)\mu
 \end{align} (see  Lemma 3.1 in \cite{LX} or Proposition 2.1 in \cite{WX1} for a proof).  These two simple facts are important in the study of  capillary hypersurfaces.
 From these two face we have
 \begin{equation}\label{tang killing sec F.F.}
 \begin{split}
 h(X_a^T,\mu)& =\langle D_{\mu}\nu,X_a^T\rangle=h(\mu,\mu) \langle\mu,X_a^T\rangle=h(\mu,\mu) \langle \mu,\langle x,a\rangle x-a\rangle 
 \\&=h(\mu,\mu) \langle x,a\rangle \sin\theta-h(\mu,\mu)\langle\mu,a\rangle
 \\&=-h(\mu,\mu)\langle a,\cos\theta \overline{\nu}\rangle,
 \end{split}
 \end{equation} where we have used the fact $\mu=\sin\theta \overline{N}+\cos\theta \overline{\nu}$ in the last equality.
 
 The following proposition was proved for hypersurfaces with free boundary recently in \cite{WX1}. For completeness, we provide a proof here for hypersurfaces with  capillary boundary.
 
 \begin{prop}\label{sigma 2 cap}Let $x:M\to \overline{\mathbb{B}}^{n+1}$ be an embedded smooth hypersurface in $\mathbb{B}^{n+1}$ with capillary boundary of contact angle $\theta\in (0,\pi)$. Then
 	
 	\begin{eqnarray*}\nonumber
 		\int_{M} H\langle x,a\rangle dA&=&
 		\frac{2}{n-1} \int_{M} \sigma_2(\kappa) \langle X_a,\nu\rangle dA \\
 		&& +\frac{1}{n-1}\int_{\partial M} \big(H \langle X_a^T,\mu\rangle-h(X_a^T,\mu)\big)d\sigma,
 	\end{eqnarray*}
 	where $dA$ and $d\sigma$ are the area element of $M$ and $\partial M$ respectively with respect to the induced metric $g$, $\kappa:=(\kappa_1,\ldots,\kappa_n)$ are the principal curvatures of the Weingarten tensor $(g^{-1}h)$ and $\sigma_2(\kappa)$ is the $2$nd elementary symmetric function acting on the principal curvatures.
 \end{prop}
 
 \begin{proof} Let $\{e_{i}\}_{i=1}^n$ be the orthonormal frame on $M$ and $e_{n+1}=\overline{N}$. 
 	By using equation (3.5) in \cite{WX1}, we have
 	\begin{equation*}%\nonumber
 	\frac{1}{2} \Big( \nabla_{i} (X_a^T)_{j}+\nabla_{j} (X_a^T)_{i}\Big)=\langle x,a \rangle g_{ij}-h_{ij}\langle X_a,\nu\rangle.
 	\end{equation*} This follows easily from the conformality of the vector field $X_a$. It follows that
	\begin{equation}\label{eq_a2}
	{\rm div}\, X_a^T= n\langle x, a \rangle-H\langle X_a,\nu\rangle.
	\end{equation}
 	Denote  the Newton tensor by $T_1(\kappa):=\frac{\partial\sigma_2}{\partial (g^{-1}h)}$ . In local coordinates, we have 
 	$T_1^{ij}:=\frac{\partial \sigma_2}{\partial h_{j}^{i}}$.
 	Multiplying    the both side of the above identity by $T_1^{ij}:=\frac{\partial \sigma_2}{\partial h_{j}^{i}}$ and integrating, we have 
 	\begin{eqnarray*}
 	\int_{M} T^{ij}_1(\kappa) \nabla_{i} (X_a^T)_{j}dA
 	&=& \ds\vs \int_{M} (Hg_{i j}-h_{ji} ) \cdot (\langle x,a \rangle g_{i j}-h_{i j}\langle X_a,\nu\rangle  )dA \\
	&=&\ds\vs \int_{M}  ( (n-1)H\langle x,a\rangle-(H^2-|h|^2)\cdot\langle X_a,\nu\rangle )dA \\
	&=&\ds \int_{M} \big ((n-1) H\langle x,a\rangle-2\sigma_2(h) \langle X_a,\nu\rangle \big)dA .
 	\end{eqnarray*}
 	Since   $X_a^T$ is the tangential projection of $X_a$ on $M$,  integrating by parts we have 
 	\begin{eqnarray*}\nonumber
 		\int_{M} T^{ij}_1(\kappa) \nabla_{i} (X_a^T)_{j}dA
 		&=& \int_{\partial M} T_1^{ij} (X_a^T)_{j} \mu_{i}d\sigma=\int_{\partial M} T_1(X_a^T,\mu)d\sigma
 		\\
 		&=& \int_{\partial M} \Big([H \langle X_a^T,\mu\rangle-h(X_a^T,\mu)\Big )d\sigma.
 	\end{eqnarray*}
 	Hence the proof is complete.
 \end{proof}
 The following property is also crucial for us.
 \begin{prop}\label{surface and bdry relation with H} Under the same conditions as in Proposition \ref{sigma 2 cap}, it holds that
 	\begin{equation*}
 	(n-1)\int_{M} H\langle \nu,a\rangle dA=\int_{\partial M}\big (H-h(\mu,\mu)\big )\langle \overline{\nu},a\rangle d\sigma.	\end{equation*}
 \end{prop}
 
 \begin{proof}
 	Set $P_a:=\langle \nu,a\rangle x-\langle x,\nu\rangle a$.  By a direct computation, we have
 	\begin{equation}\label{eq_a31}%\nonumber
 	\begin{split}
 	\nabla_{e_{j}}  \langle P_a,e_{i}\rangle  
 	&= \nabla_{e_{j}} \big( \langle \nu,a\rangle \langle x,e_{i}\rangle-\langle x,\nu\rangle \langle a,e_{i}\rangle\big)
 	\\&=\langle h_{jk}e_{k},a\rangle \langle x,e_{i}\rangle+\langle \nu,a\rangle \delta_{ij}+\langle \nu,a\rangle \langle x,-h_{ij}\nu\rangle
 	\\&\quad -\langle x,h_{jk}e_{k}\rangle\langle a,e_{i}\rangle+\langle x,\nu\rangle \langle a,h_{ij}\nu\rangle
 	\\&=\langle\nu,a\rangle \delta_{ij}+ h_{jk}a^{k}x^{i}-h_{jk}x^{k}a^{i}
 	\end{split}
 	\end{equation}
	and
	\begin{equation}\label{eq_a3}
	{\rm div}\, P^T_a=n \langle \nu, a\rangle.
	\end{equation}
 	Multiplying \eqref{eq_a31}  by $T_1^{ij}:=\frac{\partial \sigma_2}{\partial h_{j}^{i}}$, we obtain
 	\begin{equation}\nonumber
 	\begin{split}
 	T_1^{ij}(h)\cdot  {\nabla}_{e_{j}}   \langle P_a,e_{i}\rangle 
 	&= \big[H\delta_{ij}-h_{ji}\big]\cdot \big[\langle\nu,a\rangle \delta_{ij}+ h_{j k}a^{k}x^{i}-h_{jk}x^{k}a^{i}\big]
 	\\&=(n-1)H\langle\nu,a\rangle.
 	\end{split}
 	\end{equation}
 	 Integrating by parts we conclude that
 	\begin{equation}\nonumber
 	\begin{split}
 	(n-1)\int_{M} H\langle \nu,a\rangle dA   &=\int_{M}T_1^{ij}(h)\cdot  {\nabla}_{e_{j}}   \langle P_a,e_{i}\rangle  dA
 	\\&= \int_{\partial M} T_1^{ij}(h)\langle P_a,e_{i}\rangle \langle \mu,e_{j}\rangle d\sigma
 	\\&= \int_{\partial M} \big[H\delta_{ij}-h_{ji}\big]\langle P_a,e_{i}\rangle \langle \mu,e_{j}\rangle d\sigma
 	\\&=\int_{\partial M} \Big(H\langle \nu,a\rangle \langle x,\mu\rangle-H\langle x,\nu\rangle \langle a,\mu\rangle-h(\mu,x^T)\langle \nu,a\rangle+h(\mu,a^T)\langle x,\nu\rangle\Big)d\sigma
 	%\\&=\int_{\partial M}H\langle \nu,a\rangle \langle x,\mu\rangle-H\langle x,\nu\rangle \langle a,\mu\rangle-h(\mu,\mu)\langle x,\mu\rangle \langle \nu,a\rangle+h(\mu,\mu)\langle a,\mu\rangle \langle x,\nu\rangle
 	\\&=\int_{\partial M}\big[H-h(\mu,\mu)\big]\big[\langle \nu,a\rangle \langle x,\mu\rangle-\langle x,\nu\rangle \langle a,\mu\rangle\big]d\sigma
 	\\&=\int_{\partial M}\big[H-h(\mu,\mu)\big]\langle \overline{\nu},a\rangle d\sigma,
 	\end{split}
 	\end{equation} where we have used  equation \eqref{vanish 2nd F.F. of e,mu} in the fifth equality.
 	Therefore we complete the proof.
 \end{proof}
 
 \subsection{The first variation formulas}
 Let $x:(M,\partial M)\to (\overline{\mathbb{B}}^{n+1}, \partial\mathbb{B}^{n+1})$ be an isometric embedded of an orientable $n$-dimensional compact manifold $M$ with smooth boundary $\partial M$ such that  $\Sigma:=x(M)$ and
 $\partial\Sigma:=x(\partial M)$. We define the volume  functional of $x$ as the usual volume  of the $n+1$-dimensional domain $\Omega$  
 enclosed by $\Sigma$ and $\partial \mathbb{B}^{n+1}$  as in Figure 1.
 The so-called wetting area $W(\Sigma) $ is just the area of 
 the region $T:=\partial\Omega\cap \partial \mathbb{B}^{n+1}$, which is also bounded by $\partial \Sigma$ on $\partial \mathbb{B}^{n+1}$. The energy functional is defined as
 \begin{align}\label{total energy}
 E(x)=E(\Sigma):= \mbox{Area}(\Sigma)-\cos\theta\,  \mbox{Area}(T).
 \end{align}
 
Next we present the first variational formula for the energy functional $E$.
An admissible variation of $x$ is a differential map 
 $x: M\times(-\varepsilon,\varepsilon)\to\overline{\mathbb{B}}^{n+1}$ satisfying  that 
 $x_t(\cdot):=x(\cdot,t):M\to\overline{\mathbb{B}}^{n+1}$ is an immersion with 
 $x(\mbox{int}(M),t)\subset \mathbb{B}^{n+1}$ and $x(\partial M, t)\subset \partial\mathbb{B}^{n+1}$, and  $x(\cdot,0)=x_0(\cdot)$. 
 Denote the corresponding hypersurfaces by $\Sigma_t=x(M,t)$, its enclosed domain $\Omega_t$ and the ``wet'' part by $T_t$.
It is well-known that the first variations of volume functional and area functional are given by
 \begin{equation}\nonumber
 \frac{d}{dt} \mbox{Vol}(\Omega_t)= \int _M \langle 
 Y, \nu \rangle dA
 \end{equation}
 and
 \begin{equation}\nonumber
 \frac{d}{dt} \mbox{Area}(\Sigma_t)=\int_{M}H\langle Y,\nu\rangle dA+\int_{\partial M} \langle Y,\mu\rangle d\sigma,
 \end{equation}
 where $dV_{ \mathbb{B}^{n+1}}$ is the volume element of $\mathbb{B}^{n+1}$ and $Y:=\frac{\partial}{\partial t} x_t(\cdot)\big|_{t=0}.$
 Moreover the variation of the area of $T_t$ is given by 
 %The so called {\it wetting area functional} $W:M\times (-\varepsilon,\varepsilon) \to\mathbb{R}$ is defined by
 %\begin{align*}
 %W(M_t):=\int_{[0,t]\times \partial M} x^*(dA_{\partial \mathbb{B}^{n+1}}),
 %\end{align*}
 %and its first variation formula is
 \begin{equation}\nonumber
 \frac{d}{dt} \mbox{Aera} (T_t)= \int_{\partial M} \langle Y,\overline{\nu}\rangle d\sigma,
 \end{equation}
 %where $dA_{\partial \mathbb{B}^{n+1}}$ is the area element of $\partial \mathbb{B}^{n+1}$.
 For a proof, see \cite{RS} (See Section 4 Appendix there) for instance.
 %The energy functional $E:M\times(-\varepsilon,\varepsilon) \to\mathbb{R}$, which we are interested in,  is defined by
 %\begin{align}\label{E}
 %E(M_t):=\mbox{Area}(M_t)-\cos\theta W(M_t).
 %\end{align}
 %Hence, it follows from above that its first variation formula is
 Now, the variation of the enery functional $E$ is given by
 \begin{align} \label{2.4}
 \frac{d}{dt}E(\Sigma_t)=\int_{M}H\langle Y,\nu\rangle dA+\int_{\partial M} \langle Y,\mu-\cos\theta \,\overline{\nu}\rangle d\sigma.
 \end{align}

 \subsection{Key properties of  flow \eqref{MCF with angle} }
 From the Minkowski type formula in \cite{WX1} (see Proposition 3.2 and equation (3.4) there), we have the following two  important facts of \eqref{MCF with angle}.
 
 \begin{prop}\label{energy de}  Flow \eqref{MCF with angle} preserves the volume functional $\rm{Vol}(\Omega_t)$ and decreases  $E(M_t)$.
 \end{prop}

\begin{proof} 
 It is easy to see that this flow preserves the enclosed volume $\Omega_t$ of $x(M,t)$ in $\mathbb{B}^{n+1}$, since 
 \begin{equation} 
\begin{split}
 \frac{d}{dt} \mbox{Vol}(\Omega_t)&=\frac{d}{dt}\int_{[0,t]\times M} x^*dV_{\mathbb{B}^{n+1}} = \int_M f dA
 %\\&=\int_M \langle f\nu+f\cot\theta \overline{\mu},\nu\rangle dA
 \\&=\int_{M}\Big[n\langle x,a\rangle+n\cos\theta \langle \nu,a\rangle -H\langle X_a,\nu\rangle\Big]dA=0,
% \\&=\int_M \mbox{div}(X_a^T+\cos\theta P_a^T)dA
 %\\&=\int_{\partial M} \langle X_a^T+\cos\theta P_a^T,\mu\rangle d\sigma= 0,
\end{split}
 \end{equation}
 where the last equality is the new Minkowski identity proved in \cite{WX1}. With the above preparation and for the convenience of the reader, we point out that this formula
 follows from 
 \[f= {\rm div}\, (X_a^T+\cos \theta P^T_a)\quad  \text{ in }M, \quad  \langle X_a^T+\cos \theta P^T_a, \mu\rangle=0, \quad \text{ on } \partial M.
 \]
which, in turn, follows from 
equations \eqref{eq_a2}, \eqref{eq_a3} and  \eqref{tang killing metric}.

 %Furthermore,  we have the following monotonicity  of the energy functional.

 	From \eqref{2.4} and Proposition 2.2, we have  that
 	\begin{equation}\label{energy deri}
 	\begin{split}
 	\frac{d}{dt} E(M_t)&:=\frac{d}{dt} \Big[ \mbox{Area}(M_t)-\cos\theta W(M_t)\Big]
 	%\\ &=\int_{M} HfdA + \int_{\partial M} \langle f\nu+V,\mu-\cos\theta \overline{\nu}\rangle d\sigma
 	%\\&=\int_{M} HfdA + \int_{\partial M} \langle G\overline{\nu},\sin\theta \overline{N}\rangle d\sigma
 	\\& = \int_{M} H\big(n\langle x,a\rangle+n\cos\theta \langle \nu,a\rangle-H\langle X_a,\nu\rangle\big)dA
 	\\&=\Big[n \int_{M} \cos\theta H\langle \nu,a\rangle dA +\frac{n}{n-1}\int_{\partial M} \big(H \langle X_a^T,\mu\rangle -h(X_a^T,\mu)\big)d\sigma\Big]
 	\\&\quad  +
 	\int_{M} {\big( \frac{2n}{n-1} \sigma_2(\kappa)-|H|^2\big)} \langle X_a,\nu\rangle dA
 	\\&:=\text{S}_1+\text{S}_2.
 	\end{split}
 	\end{equation}
 	For the term $\text{S}_2$, we claim that $\text{S}_2\leq 0$. In fact, this follows from facts that $\langle X_a,\nu\rangle> 0$ in $M$ and the following well-known fact 
 	\begin{equation}
 	\begin{split}
 	\frac{2n}{n-1}\sigma_2(\kappa)-H^2&=\frac{1}{n-1}\big[2\sigma_2(\kappa)-(n-1)\sum_{i=1}^{n} \kappa^2_{i}\big]\\&=-\frac{1}{n-1}\sum_{1\leq i<j\leq n} (\kappa_{i}-\kappa_{j})^2 \leq 0.
 	\end{split}
 	\end{equation}
 	For the term S$_1$, from equations \eqref{tang killing metric} and \eqref{tang killing sec F.F.}, we have
 	\begin{equation}\nonumber
 	\begin{split}
 	h(X_a^T,\mu)=-\cos\theta h(\mu,\mu)\langle a,\overline{\nu}\rangle,\quad \langle X_a^T,\mu\rangle =\langle X_a,\mu\rangle = -\cos\theta \langle a,
 	\overline{\nu}\rangle.
 	\end{split}
 	\end{equation}
 	Combining with Proposition \ref{surface and bdry relation with H} and the fact that  $\theta\equiv const$, we have 
 	\begin{eqnarray*}
 		\frac{\text{S}_1}{n}&=&\int_{M} \cos\theta H\langle \nu,a\rangle dA+\frac{1}{n-1}\int_{\partial M} \big(H \langle X_a^T,\mu\rangle-h(X_a^T,\mu)\big)d\sigma
 		\\&= &  \int_{M} \cos\theta H\langle \nu,a\rangle dA-\frac{\cos\theta}{n-1}\int_{\partial M}  \big[H-h(\mu,\mu)\big]\langle a,\overline{\nu}\rangle d\sigma
 		=0.
 	\end{eqnarray*}
 	Therefore, we obtain
 	\begin{equation*}
 	\begin{split}
 	\frac{d}{dt} E(M_t)&:=\frac{d}{dt}\Big[ \mbox{Area}(M_t)-\cos\theta W(M_t)\Big]
 	=\text{S}_1+\text{S}_2\leq 0.
 	\end{split}
 	\end{equation*}
 	Hence we complete the proof.
 \end{proof}

 \section{A scalar equation}
 In this section we will reduce flow \eqref{MCF with angle} to a scalar flow, provided the initial hypersurface is  star-shaped.
 
 \subsection{Basic facts}
 In this subsection, we first recall some basic facts and identities for the relevant geometric quantities of a smooth 
 star-shaped hypersurface  $X:M\to \Sigma\subset \mathbb{R}^{n+1}_+$ 
 with respect to the origin. %We assume the origin is not on $\Sigma$. 
 If $\Sigma$ is star-shaped with respect to the origin, then the position vector $X$ of $\Sigma$ can be written
 as in \begin{equation*}
 X:=e^{u(x)}x=\rho(x)x \quad x\in\Omega\subset\mathbb{S}^n_+,
 \end{equation*} where $u\in C^2(\Omega)\cap C^0(\overline{\Omega})$ and $\rho:=e^u$.
 
 Let $\{e_i\}_{i=1}^n$ be the local frame field on $\mathbb{S}^n_+$ with the round metric $\sigma$, and denote $\nabla$ and $D$ 
  the gradient on $\mathbb{S}^n_+$ and $\mathbb{R}^{n+1}_+$ respectively. % and covariant differentiation will simply be indicated by indices. 
 Then in terms of $\rho$ the metric $g_{ij}$ is  given by
 \begin{equation*}
 g_{ij}=\langle D_{e_i}X,D_{e_j}X\rangle=e^{2u}\big(\sigma_{ij}+u_iu_j\big)=\rho^2\sigma_{ij}+\rho_i\rho_j,
 \end{equation*}
 where $\langle \cdot,\cdot\rangle$ denotes the standard inner product in $\mathbb{R}^{n+1}_+$,  $\sigma_{ij}:=\langle e_i,e_j\rangle$ and
 $\rho_i:=\nabla_{e_i}\rho$, $ \rho_{ij}:=\nabla_{e_i}\nabla_{e_j}\rho$. The inverse of metric $g$ is
 \begin{equation*}
 g^{ij}=e^{-2u}\big(\sigma^{ij}-\frac{u^iu^j}{1+|\nabla u|^2}\big)=\rho^{-2}\big( \sigma^{ij}-\frac{\rho^i\rho^j}{\rho^2+|\nabla \rho|^2}\big),
 \end{equation*}
 where $\sigma^{ij}$ denotes the inverse of $\sigma_{ij}$ and $u^i:=\sigma^{ik}u_k$.
 The unit outer normal vector field to  $\Sigma$ %$X$ 
 in $\mathbb{R}^{n+1}_+$ is given by
 \begin{equation*}
 \nu(X(x))=\frac{x-\nabla u(x)}{\sqrt{1+|\nabla u|^2}}=\frac{x\rho-\nabla \rho(x)}{\sqrt{\rho^2+|\nabla \rho|^2}}.
 \end{equation*}
 Note that $\langle \nu,X\rangle=\frac{\rho^2}{\sqrt{\rho^2+|\nabla \rho|^2}}=\frac{e^u}{\sqrt{1+|\nabla u|^2}} >0$ which means that $\nu$ satisfies the choice of orientation on a radial graph.
 The second fundamental form of $X$ is
 \begin{equation*}
 h_{ij}=-\langle D_{e_i}D_{e_j}X,\nu\rangle=e^u\frac{\sigma_{ij}+u_iu_j-u_{ij}}{\sqrt{1+|\nabla u|^2}}=-\frac{\rho\rho_{ij}-\rho^2\sigma_{ij}-2\rho_i\rho_j}{\sqrt{\rho^2+|\nabla \rho|^2}},
 \end{equation*} 
 and the mean curvature is given by 
 \begin{equation*}
 \begin{split}
 H&:=\sum_{i,j=1}^{n}g^{ij}h_{ij}= \frac{e^{-u}}{\sqrt{1+|\nabla u|^2}} \Big( n-\Delta u+\sum_{i,j=1}^{n}\frac{u_{ij}u^iu^j}{1+|\nabla u|^2}\Big)
 \\&=-e^{-u}\mbox{div}_{_{\mathbb{S}^n_+}} \Big(\frac{\nabla u}{\sqrt{1+|\nabla u|^2}}\Big)+\frac{ne^{-u}}{\sqrt{1+|\nabla u|^2}}
 \\&=-\frac{1}{\rho}\mbox{div}_{_{\mathbb{S}^n_+}}\big(\frac{\nabla \rho}{\sqrt{\rho^2+|\nabla \rho|^2}}\big)+ \frac{n}{\sqrt{\rho^2+|\nabla \rho|^2}}
 \\&=\frac{n}{\rho v}-\frac{1}{\rho v}\sum_{i,j=1}^{n}\big(\sigma^{ij}-\frac{\rho^i\rho^j}{v^2}\big)\rho_{ij},
 \end{split}
 \end{equation*}
 where $v:=\sqrt{1+|\nabla u|^2}$ and $\mbox{div}_{_{\mathbb{S}^n_+}}$ is the divergence operator with respect to the canonical metric $\sigma$ on $\mathbb{S}^n_+$.
 
 Using the same method in \cite{Ge}, we assume that  flow equation \eqref{MCF with angle} is satisfied by a family of  the radial graphs over $\mathbb{S}^n_+$,
 that is, $x(\xi,t):=X(\xi,t)\rho(X(\xi,t),t)$ with $X\in\mathbb{S}^n_+$.  Then we have
 \begin{equation}\label{scalar equation}
 \begin{split}
 f&= \langle \frac{\partial x}{\partial t},\nu\rangle
 \\&= \langle \frac{\partial X}{\partial t}\rho+X\cdot \big(\nabla \rho\cdot \partial_tX\big)+X\partial_t \rho,\frac{X\rho-\nabla\rho}{\sqrt{|\nabla \rho|^2+\rho^2}}\rangle
 \\&=\frac{\partial \rho}{\partial t}\cdot \frac{\rho}{\sqrt{|\nabla \rho|^2+\rho^2}}=\frac{1}{\sqrt{1+|\nabla u|^2}}\frac{\partial u}{\partial t}.
 \end{split}
 \end{equation}
 %where we have used the fact that $\langle \overline{\mu},\nu\rangle=0$.
 
 \subsection{A conformal transformation}
 We use the following coordinate transformation $\varphi$ as in \cite{WX2} to transform the unit ball into the half space
 \begin{equation*}
 \begin{split}
 \varphi: \quad & \mathbb{B}^{n+1}\longrightarrow \mathbb{R}^{n+1}_+
 \\& (x,x_{n+1}) \longmapsto \frac{2x+(1-|x|^2-x_{n+1}^2)e_{n+1}}{|x|^2+(x_{n+1}-1)^2}:=(y,y_{n+1}),
 \end{split}
 \end{equation*}
 where $x:=(x_1,\ldots,x_n)\in\mathbb{R}^n$.
 Equivalently we have
 \begin{equation*}
 \begin{array}{rcll}
 x_i &=&\ds\vs \frac{2y_i}{|y|^2+(y_{n+1}+1)^2}, &\quad 1\leq i\leq n,\\
 x_{n+1}&=& \ds \frac{|y|^2+y_{n+1}^2-1}{|y|^2+(1+y_{n+1})^2}.
 \end{array}
 \end{equation*} 
 Moreover, $\varphi$ maps $\mathbb{S}^n=\partial\mathbb{B}^{n+1}$ to $\partial \mathbb{R}^{n+1}_+:=\{(y,y_{n+1})\in\mathbb{R}^{n+1}:y_{n+1}=0\}$.  By a direct computation, one gets
 \begin{equation*}
 \varphi^*(\delta_{\mathbb{R}^{n+1}_+})=\frac{4}{\big(|x|^2+(x_{n+1}-1)^2\big)^2} \delta_{\mathbb{B}^{n+1}},
 \end{equation*}
 which means that $\varphi$ is a conformal transformation from $\big(\mathbb{B}^{n+1},\delta_{\mathbb{B}^{n+1}}\big)$ to $\big(\mathbb{R}^{n+1}_+,\delta_{\mathbb{R}^{n+1}_+}\big) $. (Another view to see this fact is that it comes from the M\"{o}bius transformation
 $ M(z):=\frac{1-iz}{z-i},
 $ and rotational symmetry with $z:=|x|+x_{n+1} i$.)
 %(Another view to see this is that it comes from the M\"{o}bius transformation $ M(z):=\frac{1-iz}{z-i},$ and rotational symmetry with $z:=|x|+x_{n+1} i$.)
 
 We define $X_{n+1}$ to the
conformal vector field $X_a$ with $a=-E_{n+1}$, that is,
 $X_{n+1}:=-\langle \tilde{x},E_{n+1}\rangle \tilde{x}+\frac{|\tilde{x}|^2+1}{2}E_{n+1}$, where $E_{n+1}$ is the  standard $(n+1)$-th component vector field in $\mathbb{B}^{n+1}$ and $\tilde{x}:=(x,x_{n+1})\in\mathbb{B}^{n+1}$.
 One can directly compute to find that $$ {\varphi_*(X_{n+1})=(y,y_{n+1}):=\tilde{y}}\quad \text{in}\quad \mathbb{R}^{n+1}_+.$$
 For a   hypersurface $\Sigma\subset\overline{\mathbb{B}}^{n+1}$ with capillary boundary $\partial\Sigma\subset\mathbb{S}^n$, we have
 \begin{equation*}
 \frac{4}{\big[|x|^2+(x_{n+1}-1)^2\big]^2} \langle X_{n+1},\nu\rangle=\langle \varphi_*(X_{n+1}),\varphi_*(\nu)\rangle=|\varphi_*(\nu)| \langle \tilde{y},\tilde{\nu}\rangle,
 \end{equation*}
 where $|\varphi_*(\nu)|= \frac{|y|^2+(y_{n+1}+1)^2}{2}$ and $\tilde{\nu}:=\frac{\varphi_*(\nu) }{|\varphi_*(\nu)|}$.
 Hence the hypersurface $\varphi(\Sigma)$ is star-shaped in $\mathbb{R}^{n+1}_+$ with respect to the origin, i.e.,  $\langle \tilde y, \tilde \nu\rangle>0$  if and only if $\langle X_{n+1},\nu\rangle > 0$ holds on $\Sigma$.
 Therefore, under the transformation  flow \eqref{MCF with angle}  is equivalent to
 \begin{equation}\label{MMCF with tangential variation after conformal tran}
 \begin{array}{rcll}
 \partial_t \tilde{y}&=&\varphi_*(\partial_t \tilde{x})=\big(\tilde{f}\cdot|\varphi_*(\nu)|\big)\tilde{\nu} 
 & \hbox{ in }\varphi(\Sigma)\times[0,T), \\
 \langle \tilde{\nu},\tilde{N} \rangle &=&\cos\theta\
 & \hbox{ on }\varphi(\partial\Sigma)\times [0,T),\\
 \tilde{y}(0) & =&\varphi(\tilde{x}(0))=\varphi(\tilde{x}_0):=\tilde{y}_0 \quad &   \text{ on} \quad \varphi(\overline{\Sigma})\times\{0\}, 
 
 \end{array}
 \end{equation}
 where $\tilde{f}:=f\circ\varphi^{-1}$,  $\tilde{N}:=\frac{\partial}{\partial y_{n+1}}$ is the inner normal vector field of $\varphi(\partial\Sigma)\subset \mathbb{R}^n\times \{0\}$ in $\mathbb{R}^{n+1}_+$.
 
 Now in $\mathbb{R}^{n+1}_+$, we use the polar coordinate $(\rho,\beta,\xi)\in [0,+\infty)\times [0,\frac{\pi}{2}]\times \mathbb{S}^{n-1}$ as in \cite{WX2}, where $\xi$ is the spherical coordinate in $\mathbb{S}^{n-1}$ and 
 \begin{equation*}
 \begin{cases}
 \rho^2=|y|^2+y_{n+1}^2,\\ y_{n+1}=\rho\cos\beta, |y|=\rho \sin\beta.
 \end{cases}
 \end{equation*}
 Then it implies that the standard Euclidean metric in $\mathbb{R}^{n+1}_+$ has the expression 
 \begin{equation}
 \begin{split}
 \delta_{\mathbb{R}^{n+1}_+} &=|d\tilde{y}|^2 =d\rho^2+\rho^2g_{_{\mathbb{S}^n_+}}
 \\&= d\rho^2+\rho^2(d\beta^2+\sin^2\beta g_{_{\mathbb{S}^{n-1}}}),
 \end{split}
 \end{equation}
 where $g_{_{\mathbb{S}^{n-1}}}$ is the standard spherical metric on $\mathbb{S}^{n-1}$. Since $\big(\mathbb{B}^{n+1},\delta_{\mathbb{{B}}^{n+1}}\big)$ and $\big(\mathbb{R}^{n+1}_+,(\varphi^{-1})^*(\delta_{\mathbb{R}^{n+1}_+})\big)$ are isometric, a proper embedding $\Sigma=\tilde{x}(M)$ in $\big(\mathbb{B}^{n+1},\delta_{\mathbb{{B}}^{n+1}}\big)$ can be identified as $\widetilde{\Sigma}$ in $\big( \mathbb{R}^{n+1}_+,(\varphi^{-1})^*(\delta_{\mathbb{R}^{n+1}_+})\big)$.
  
  For a star-shaped hypersurface $\widetilde{\Sigma}:=\tilde{y}(M)$ in $\big(\mathbb{R}^{n+1}_+,(\varphi^{-1})^*(\delta_{\mathbb{R}^{n+1}_+})\big)$, 
 where $\tilde{y}:=\varphi\circ\tilde{x}$, we can write it as
 \begin{align*}
 \tilde{y}=\rho (z)z=\rho(\beta,\xi)z,\quad z:=(\beta,\xi)\in\overline{\mathbb{S}}^n_+.
 \end{align*}
 In polar coordinates, a direct computation implies that
 \begin{equation*}
 \frac{\partial}{\partial y_{n+1}}=\frac{\partial \rho}{\partial y_{n+1}} \partial_{\rho}+\frac{\partial \beta}{\partial y_{n+1}} 
 \partial_{\beta}=\cos\beta \partial_{\rho}-\frac{\sin\beta}{\rho}\partial_{\beta}.
 \end{equation*}
It  gives us  
 \begin{equation*}
 \begin{split}
 \sum_{i=1}^{n}y_i\partial_{y_i}=\rho\partial_{\rho}-\rho \cos\beta\big(\cos\beta \partial_{\rho}-\frac{\sin\beta}{\rho}\partial_{\beta}\big)
 =\rho\sin^2\beta \partial_{\rho}+\frac{\sin 2\beta}{2}\partial_{\beta}.
 \end{split}
 \end{equation*}
% Then we deduce \begin{equation*} \begin{split} \varphi_*(\tilde{x})&=\varphi_*(\sum_{\alpha=1}^{n+1} x_{\alpha}\frac{\partial }{\partial x_{\alpha}})=\sum_{\alpha,\beta=1}^{n+1}x_{\alpha}\frac{\partial y^{\beta}}{\partial x_{\alpha}}\frac{\partial}{\partial y_{\beta}} \\&=\sum_{i=1}^{n}y_{n+1}y_i\frac{\partial}{\partial y_i}+\frac{y_{n+1}^2-|y|^2-1}{2}\frac{\partial}{\partial y_{n+1}} \\&=\rho \cos\beta\big(\rho\sin^2\beta\partial_{\rho}+\frac{\sin 2\beta}{2}\partial_{\beta}\big)+\frac{\rho^2\cos 2\beta-1}{2}\big(\cos\beta\partial_{\rho}-\frac{\sin\beta}{\rho}\partial_{\beta}\big) \\&=\frac{\rho^2-1}{2}\cos\beta\partial_{\rho}+\frac{\rho^2+1}{2\rho}\sin\beta\partial_{\beta}. \end{split} \end{equation*} In particular, this gives $\varphi_*(\tilde{x})=\frac{\rho^2+1}{2\rho} \partial_{\beta}$ on $\partial\mathbb{S}^n_+$.
 From now on, we always set $a:=-E_{n+1}$. We have 
 \begin{equation*}
 \begin{split}
 -\varphi_*(a)&=\sum_{i=1}^{n}\frac{\partial y_{i}}{\partial x_{n+1}}\frac{\partial}{\partial y_i}+\frac{\partial y_{n+1}}{\partial x_{n+1}}\frac{\partial}{\partial y_{n+1}}
 % \\&=\frac{-4(x_{n+1}-1)\cdot x_i}{A^2}\frac{\partial}{\partial y_i}+\frac{-2x_{n+1}A-2(x_{n+1}-1)(1-|\tilde{x}|^2)}{A^2}\frac{\partial}{\partial y_{n+1}}
 \\&=(1+y_{n+1})\sum_{i=1}^n y_i \frac{\partial}{\partial y_i}+\frac{(1+y_{n+1})^2-|y|^2}{2}\frac{\partial}{\partial y_{n+1}} \\&=y_{n+1}\sum_{i=1}^{n}y_i\partial_{y_i}+\frac{1+y_{n+1}^2-|y|^2}{2}\partial_{y_{n+1}}+\sum_{\alpha=1}^{n+1}y_{\alpha}\partial_{y_{\alpha}}
 \\&=\rho\cos\beta\big[\rho\sin^2\beta \partial_{\rho}+\frac{\sin 2\beta}{2}\partial_{\beta}\big]+\frac{1+\rho^2\cos 2\beta}{2}\big(\cos\beta \partial_{\rho}-\frac{\sin\beta}{\rho}\partial_{\beta}\big)+\rho\partial_{\rho}
 \\&=\frac{\rho^2\cos\beta+2\rho+\cos\beta}{2}\partial_{\rho}+\frac{(\rho^2-1)\sin\beta}{2\rho}\partial_{\beta}.
 \end{split}
 \end{equation*}
 Set $w:=\log \frac{2}{|y|^2+(y_{n+1}+1)^2}=\log \frac{2}{\rho^2+2\rho \cos\beta+1}$ and  $u:=\log \rho$.  From the discussion in Section 3.1, we know that $\tilde{\nu}=\frac{\partial_{\rho}-\rho^{-1}\nabla u}{v}$ is the unit outward normal vector field of $\widetilde{\Sigma}$ in $(\mathbb{R}^{n+1}_+, \delta_{\mathbb{R}^{n+1}_+})$. Then the capillary boundary condition gives us that
 \begin{equation*}
 \begin{split}
 -e^{-2w}\cos\theta&=e^{-2w}\delta_{\mathbb{B}^{n+1}}(\nu,\overline{N}\circ \tilde{x})
 \\&= \varphi^*\delta _{\mathbb{R}^{n+1}_+}(\nu,\overline{N}\circ \tilde{x})\\&= \delta _{\mathbb{R}^{n+1}_+}(\varphi_*(\nu),\varphi_*(\tilde{x}))
 \\&=\langle e^{-w}\tilde{\nu},\frac{\rho^2+1}{2\rho}\partial_{\beta}\rangle
 =e^{-2w}\langle \frac{\partial_{\rho}-\rho^{-1}\nabla u}{v},\frac{1}{\rho}\partial_{\beta}\rangle
 \\&=-e^{-2w}\frac{\nabla_{\partial_{\beta}}u}{v},
 \end{split}
 \end{equation*}
 It follows that
 \begin{align*}
 \nabla_{\partial_{\beta}}u=\cos\theta v\quad\text{on}\quad \partial\mathbb{S}^n_+.
 \end{align*}
 By a straightforward computation as above, under the conformal transformation $\varphi$ we have
 \begin{equation*}
 \begin{split}
 \langle \nu,a\rangle
 &=e^{2w}\delta _{\mathbb{R}^{n+1}_+}(\varphi_*(\nu),\varphi_*(a))
 \\&=e^w\langle \frac{\partial_{\rho}-\rho^{-1}\nabla  u}{v},\frac{\rho^2\cos\beta+2\rho+\cos\beta}{2}\partial_{\rho}+\frac{(\rho^2-1)\sin\beta}{2\rho}\partial_{\beta}\rangle
 \\&=\frac{e^w}{2v} \big(\rho^2\cos\beta+2\rho+\cos\beta\big)-\frac{e^w(\rho^2-1)\sin\beta}{2v} \nabla _{\partial_\beta}u 
 \end{split}
 \end{equation*}
 and
 \begin{equation*}
 \langle \tilde{x},a\rangle=\langle \varphi^{-1}(\tilde{y}),a\rangle=-x_{n+1}=-\frac{|y|^2+y_{n+1}^2-1}{|y|^2+(y_{n+1}+1)^2}=-\frac{\rho^2-1}{2}e^w.
 \end{equation*}
 Similarly, we have
 \begin{equation*}
 \begin{split}
 \langle X_a,\nu\rangle&=\frac{4}{\big[|y|^2+(y_{n+1}+1)^2\big]^2} \langle \varphi_*(X_a),\varphi_*(\nu)\rangle
 \\&=\frac{2}{|y|^2+(y_{n+1}+1)^2} \langle \tilde{y},\tilde{\nu}\rangle
 \\&=e^w\langle \rho\partial_{\rho},\frac{\partial_{\rho}-\rho^{-1}\nabla u}{v}\rangle
 \\&=e^w \frac{\rho}{v}.
 \end{split}
 \end{equation*}
 Note that $e^{-w}:=\frac{\rho^2+2\rho\cos\beta+1}{2}$.  It  then yields that
 \begin{equation*}
 \begin{split}
 D_{\tilde{\nu}}e^{-w}&=\langle 
 D e^{-w},\frac{\partial_{\rho}-\rho^{-1}\nabla u}{v}\rangle
 \\&=\langle (\rho+\cos\beta)\partial_{\rho}+\rho^{-1}\partial_{\beta} e^{-w}\cdot \rho^{-1}\partial_{\beta}, \frac{\partial_{\rho}-\rho^{-1}\nabla u}{v}\rangle
 \\&=\langle (\rho+\cos\beta)\partial_{\rho}-\sin\beta \rho^{-1}\partial_{\beta}, \frac{\partial_{\rho}-\rho^{-1}\nabla u}{v}\rangle
 \\&=\frac{\rho+\cos\beta+\sin\beta \nabla_{\partial_\beta} u}{v}.
 \end{split}
 \end{equation*}
 Applying the transformation law for the mean curvature under a conformal metric, we know that
 the mean curvature $\tilde{H}$ of $\tilde{\Sigma}$ in $\big( \mathbb{R}^{n+1}_+,(\varphi^{-1})^*(\delta_{\mathbb{R}^{n+1}_+})\big)$ is given by (see \cite{WX2}, equation (13) there)
 \begin{equation*}
 \begin{split}
 \tilde{H}&=e^{-w}\big(H_{\tilde{\nu}}+nD _{\tilde{\nu}} w\big)
 \\&=e^{-w}\Big[\frac{n}{\rho v}-\frac{1}{\rho v}\sum_{i,j=1}^{n}(\sigma^{ij}-\frac{u^iu^j}{v^2})u_{ij}\Big]-nD_{\tilde{\nu}}e^{-w}
 \\&=e^{-w}\Big[\frac{n}{\rho v}-\frac{1}{\rho v}\sum_{i,j=1}^{n}(\sigma^{ij}-\frac{u^iu^j}{v^2})u_{ij}\Big]-n \frac{\rho+\cos\beta+\sin\beta \nabla_{\partial_\beta} u}{v}
 \\&=-\Big[\frac{1}{\rho ve^w} \sum_{i,j=1}^{n}(\sigma^{ij}-\frac{u^iu^j}{v^2})u_{ij}+\frac{n\sin\beta \nabla_{\partial_\beta}u}{v}+\frac{n(\rho^2-1)}{2\rho v}\Big],
 \end{split}
 \end{equation*}
 where $H_{\tilde{\nu}}$ is the mean curvature with respect to $\tilde{\nu}$ of $\widetilde{\Sigma}$ in $(\mathbb{R}^{n+1}_+, \delta_{\mathbb{R}^{n+1}_+})$.
 From the discussion in Section 3.1, in particular, equation \eqref{scalar equation}, we know that the first equation in 
 \eqref{MMCF with tangential variation after conformal tran} is reduced to  the  following scalar equation
 \begin{equation}\label{salar equation of tilde}
 \frac{\partial_t \rho}{v}= \frac{\tilde{f}}{e^w},
 \end{equation}
 where \begin{equation*}
 \begin{split}
 \tilde{f}&:=n\langle x,a\rangle+n\cos\theta \langle \nu,a\rangle -\tilde{H}\langle X_a,\nu\rangle 
 \\&=-\frac{n}{2}(\rho^2-1)e^w  +\frac{n\cos\theta}{2}\frac{e^w}{v}\big(\rho^2\cos\beta+2\rho+\cos\beta\big)\\&\quad-\frac{n\cos\theta}{2}\frac{e^w}{v} {(\rho^2-1)\sin\beta} \nabla_{\partial_\beta} u  +\big[\frac{1}{\rho ve^w} (\sigma^{ij}-\frac{u^iu^j}{v^2})u_{ij}+\frac{n\sin\beta \nabla_{\partial_\beta} u}{v}\\&\quad+\frac{n(\rho^2-1)}{2\rho v}\big]\cdot \frac{\rho e^w}{v}.
 \end{split}
 \end{equation*}
It is easy to see that  equation \eqref{salar equation of tilde} is also equivalent to
 \begin{equation*}
 \begin{split}
 \partial_t u&=\frac{v}{\rho e^w}\tilde{f}=\frac{1}{\rho v e^w}\Big(\sigma^{ij}-\frac{u^iu^j}{v^2}\Big) u_{ij}+\Big[\frac{n\sin\beta u_{\beta}}{v}-\frac{n(\rho^2-1)}{2\rho}\frac{|\nabla  u|^2}{v}\\&\quad+\frac{n\cos\theta}{2\rho}\big(\rho^2\cos\beta+2\rho+\cos\beta\big)
 -\frac{\rho^2-1}{2\rho} n\cos\theta \sin\beta u_{\beta}\Big]
 \\&=\Big[\mbox{div}_{_{\mathbb{S}^n_+}}\big(\frac{\nabla u}{\rho v e^w}\big)-\frac{n+1}{v}\sigma\big(\nabla u,\nabla (\frac{1}{\rho e^w }) \big)
 \\&\quad-\frac{n\cos\theta}{2}\cdot\frac{\rho^2-1}{\rho}\sin\beta \sigma(\nabla u,\partial_{\beta})\Big]+\frac{n\cos\theta}{2}\cdot \frac{\rho^2\cos\beta+2\rho+\cos\beta}{\rho}
 \\&:=F(\nabla^2u,\nabla u, \rho,\beta).
 \end{split}
 \end{equation*}
 In summary, from the above discussion, flow    \eqref{MCF with angle} is equivalent to (up to a tangential diffeormphism) the following scalar parabolic equation  on $\mathbb{S}^n_+$ .
 \begin{equation}\label{MMCF with capillary u}
 \begin{array}{rcll} \ds\vs
 \frac{\partial  u}{\partial t} &=& \ds F(\nabla^2 u,\nabla u, \rho ,\beta) & \quad \text{in} \quad \mathbb{S}^n_+\times[0,T),\\
 \nabla_{\partial_{\beta}}u&=& \ds\vs \cos\theta \sqrt{1+|\nabla  u|^2}  &\quad \text{on} \quad \partial\mathbb{S}^n_+\times[0,T),\\
 u(\cdot,0)&=&u_0(\cdot) & \quad \text{on}\quad \mathbb{S}^n_+,
 \end{array}
 \end{equation}
 where $u_0=\log \rho_0$, $\rho_0$ is related to the initial hypersurface  $x_0(M)$ under the transformation $\varphi$ and $F$ is defined in the previous equation.

 \section{A priori estimates}
 The short-time existence of our flow is established by the standard PDE theory, since due to our assumption of  star-shaped, 
 \begin{align*}
 \langle X_a,\nu\rangle>0,
 \end{align*} for initial hypersurface, the flow is equivalent to  the scalar flow \eqref{MMCF with capillary u}.
 In this section, we will show the uniform height and gradient estimates for  equation \eqref{MMCF with capillary u}. 
 Then the longtime existence of the flow  follow immediately from the standard parabolic PDE theory. 
 
 In this section, we use the Einstein summation convention, i.e.,  if not stated otherwise, the repeated arabic indices $i,j,k$ should be summed from $1$ to $n$. 
 We also use the notations $u_{\beta}:=\sigma( \nabla u,\partial_{\beta})=\nabla_{\partial_\beta}u$ and $|\nabla u|^2:=\sigma(\nabla u,\nabla u)$ in this section.
 Recall that $\rho=e^u$ and $2e^{-w}={1+\rho^2+2\rho\cos\beta }$. 
 For the convenience, we introduce  the following notations
 \begin{align*}
 F^{ij}&:=\frac{\partial F(r,p,\rho,\beta)}{\partial r_{ij}}\bigg|_{r=\nabla^2 u,p=\nabla u}
 =\frac{1}{\rho ve^w} (\sigma^{ij}-\frac{u^iu^j}{v^2}), 
 \end{align*}
 \begin{align*}
 F_{p_i}&:=\frac{\partial F(r,p,\rho,\beta)}{\partial p_i}\bigg|_{r=\nabla^2 u,p=\nabla u}
 \\&=-\frac{u_i}{\rho e^w v^3}a^{kl}u_{kl}-\frac{2}{\rho v^3 e^w}a^{il}u_{kl}u_k
 +\frac{n\sin\beta \sigma( \partial_{\beta},e_i)}{v}-n\sin\beta u_{\beta}\frac{u_i}{v^3} \\&\quad-\frac{n(\rho^2-1)}{2\rho} \big(\frac{2}{v}-\frac{|\nabla u|^2}{v^3}\big)u_i-\frac{\rho^2-1}{2\rho} n\cos\theta\cos\beta \sigma( \partial_{\beta},e_i),
 \end{align*}
 \begin{align*}
 F_{\rho}&:=\frac{\partial F(r,p,\rho,\beta)}{\partial \rho}\bigg|_{r=\nabla^2 u,p=\nabla u}\\&=\frac{1}{2v}a^{ij}u_{ij}\big(1-\frac{1}{\rho^2}\big)-\frac{n}{2} \big(1+\frac{1}{\rho^2}\big) \frac{|\nabla u|^2}{v}-\frac{n}{2}\big(1+\frac{1}{\rho^2}\big) \cos\theta \sin\beta u_{\beta}\\&\quad +\frac{n\cos\theta}{2}\big(1-\frac{1}{\rho^2}\big)\cos\beta,
 \end{align*}
 \begin{align*}
 F_{\beta}&:=\frac{\partial F(r,p,\rho,\beta)}{\partial \beta}\bigg|_{r=\nabla^2 u,p=\nabla u}\\&=-\frac{1}{v}a^{ij}u_{ij}\sin\beta+\frac{n\cos\beta}{v}u_{\beta}-\cos\theta\frac{\rho^2-1}{2\rho}n\cos\beta  u_{\beta}-\frac{n\cos\theta}{2\rho} \big({\rho^2}+1\big)\sin\beta
 \end{align*}
 and
  \begin{align*}
 	a^{ij}=\sigma^{ij}-\frac{u^iu^j}{v^2},\quad\mathcal{F}:=\sum_{i=1}^n F^{ii}.
 \end{align*}
 \begin{rem}[Spherical caps]\label{static model} For any given constant $\theta \in (0, \pi)$,
% The family of spherical caps lying entirely in the closed unit ball $\mathbb{B}^{n+1}$ and meeting the support  $\mathbb{S}^{n}$ with a constant contact angle $\theta\in (0,\pi)$  are given by
 	define \begin{equation*}
 	\mathcal{C}_{r,\theta}:=\{z\in\mathbb{B}^{n+1} |z+\sqrt{r^2+2r\cos\theta+1}E_{n+1}| \leq  r\},\quad r\in(0,\infty).
 	\end{equation*}
 	It is easy to check that  $\partial \mathcal{C}_{r,\theta}$ is a static solution to the flow \eqref{MCF with angle}, that is,
 	\begin{align*}
  n\langle x,E_{n+1}\rangle+ n\cos\theta \langle \nu,E_{n+1}\rangle +H\langle X_{n+1},\nu\rangle=0
 	\end{align*}
 	and meets the support ${\mathbb S}^n=\partial \mathbb{B}^{n+1}$ at the contact angle $\theta$. Such a spherical cap is certainly star-shaped and determines  a corresponding radial function $\psi$, which is a stationary solution of
	flow \eqref{MMCF with capillary u}.
 \end{rem}
 
 Now we are ready to show that the radial function $u$ has the following $C^0$ estimate.
 \begin{prop} 
 	Assume that  the initial star-shaped hypersurface $x_0(M)$ satisfies\begin{equation*}
 	x_0(M)\subset \mathcal{C}_{R_2,\theta}\setminus \mathcal{C}_{R_1,\theta},
 	\end{equation*}
 	for some $R_2>R_1>0$, 
 	where $\mathcal{C}_{R,\theta}$ is defined in Remark \ref{static model}. Then this property is preserved along  flow \eqref{MCF with angle}. In particular, if $u(x,t)$ solves the initial boundary value problem \eqref{MMCF with capillary u} on interval $[0,+\infty)$, then for any $T>0$, 
 	\begin{equation*}
 	\|u\|_{C^0(\mathbb{S}^n_+\times[0,T])}\leq C,
 	\end{equation*}
 	where $C$ is a constant depends only on the initial value and their covariant derivatives with respect to the round metric $\sigma$ on $\mathbb{S}^n_+$.
 \end{prop}
 
 \begin{proof} For any $T>0$, we want to get the $C^0$ estimate of $u$ in $\mathbb{S}^n_+\times [0,T]$.
 	Assume that $\psi$ is the radial function of the corresponding upper spherical cap with respect to $\partial \mathcal{C}_{R_2,\theta}\cap \mathbb{B}^{n+1}$ after the conformal transformation $\varphi$. Since $\psi$ is a static solution to 
	 flow \eqref{MMCF with capillary u}, %and discussion in Section 3, 
	 we know that
 	\begin{equation*}
 	\begin{split}
 	\partial_t(u-\psi)&=F(\nabla^2 u,\nabla u,e^u,\beta)-F(\nabla^2 \psi,\nabla \psi,e^{\psi},\beta)
 	\\&= A^{ij}\nabla_{ij}(u-\psi)+ b^j\cdot (u-\psi)_j+c\cdot(u-\psi),
 	\end{split}
 	\end{equation*}
 	where $A^{ij}:=\int_0^1 F^{ij}\big(\nabla^2(su+(1-s)\psi),\nabla (su+(1-s)\psi),su+(1-s)\psi,\beta\big)ds$, $b^j:=\int_0^1 F_{p_j}\big(\nabla^2(su+(1-s)\psi),\nabla (su+(1-s)\psi),su+(1-s)\psi,\beta\big)ds$, and $c:=\int_0^1 F_\rho\big(\nabla^2(su+(1-s)\psi),\nabla (su+(1-s)\psi),su+(1-s)\psi,\beta\big) e^{su+(1-s)\psi}ds$.
 	Denote $\lambda :=-\sup\limits_{\mathbb{S}^n_+\times [0,T]} |c|$.  Applying the maximum principle, we know that $e^{\lambda t}(u-\psi)$ attains its nonnegative maximum value at the parabolic boundary, say $(x_0,t_0)$. That is,
 	\begin{equation*}
 	e^{\lambda t}(	u(x,t)-\psi(x))\leq \sup_{\partial \mathbb{S}^n_+\times [0,T)\cup \mathbb{S}^n_+\times\{0\}}\{0,e^{\lambda t} (u(x,t)-\psi(x))\}.
 	\end{equation*}
 	with either $x_0\in \partial \mathbb{S}^n_+$ or $t_0=0$.
 	If $x_0\in\partial\mathbb{S}^n_+$, from the Hopf lemma, we have
 	\begin{equation*}
 	\nabla'(u-\psi)(x_0,t_0)=0, \nabla_n u(x_0,t_0)<\nabla_n \psi(x_0,t_0),
 	\end{equation*}
 	that is, $|\nabla' u|=|\nabla' \psi|:=s$ and $\nabla_nu<\nabla_n\psi$ at $(x_0,t_0)$. Here we denote $\nabla'$ and $\nabla_n$ as the tangential and normal part of $\nabla$ on $\partial \mathbb{S}^n_+$, $e_n=-\partial_{\beta}$ is the inner normal vector field on $\partial\mathbb{S}^n_+$.
 	 From the boundary condition in \eqref{MMCF with capillary u} we have
 	\begin{equation}\nonumber
 	\frac{\nabla_nu}{\sqrt{1+s^2+|\nabla_{n}u|^2}}=-\cos\theta= \frac{\nabla_n \psi}{\sqrt{1+s^2+|\nabla_{n}\psi|^2}},
 	\end{equation}
 	 a contradiction to  the fact that function $\frac{\tau}{\sqrt{1+s^2+\tau^2}}$ is strictly increasing with respect to $\tau\in \mathbb{R}$ and $\nabla_nu<\nabla_n\psi$ at $(x_0,t_0)$.
 	Hence we have $t_0=0$, which follows that
 	\begin{equation*}
 	e^{\lambda t}(u(x,t)-\psi(x))\leq  u_0(x_0)-\psi(x_0)\leq 0, \quad \text{in}\quad (x,t)\in \mathbb{S}^n_+\times [0,T],
 	\end{equation*}
 	that is
 	\begin{equation*}
 	u(x,t)\leq \psi(x), \quad \text{in}\quad  (x,t)\in \mathbb{S}^n_+\times [0,T].
 	\end{equation*} Hence we obtain the desired upper bound of $u$.
 	Similarly, we can get the desired lower bound of $u$. After the conformal transformation we finish the proof of the Proposition.
 \end{proof}
 
 In order to obtain the gradient estimate, we need to employ the distance function $d(x):=dist_{\sigma}(x,\partial \mathbb{S}^n_+)$. It is well-known that $d$ is well-defined and smooth for $x$ near $\partial \mathbb{S}^n_+$ and $\nabla d=-\partial_{\beta}$ on $\partial \mathbb{S}^n_+$, where $\partial_{\beta}$ is the unit  {outer} normal vector field on $\partial \mathbb{S}^n_+$. In the following, we extend $d$ to be a smooth function on $\overline{\mathbb{S}}^n_+$ and satisfying that
 \begin{equation*}
 d\geq 0,\quad |\nabla d|\leq 1\quad \text{in}\quad \overline{\mathbb{S}}^n_+.
 \end{equation*}
 We will use  $O(s)$ to denote terms that are bounded by $Cs$ for  a constant $C>0$, which  depends only on the $C^0$ norm of $u$. 
 Our choice of test functions are motivated from \cite{GMWW}, \cite{Guan} and \cite{Ko}.
 Now we show the uniform gradient estimate. This is the key step of this paper.
 \begin{prop}\label{gradient estimate uniformly}
 	If $u(x,t)$ solves the initial boundary value problem \eqref{MMCF with capillary u} on the interval $[0, T^*)$ $(T^* \in (0, \infty])$ with $|\cos\theta|<\frac{3n+1}{5n-1}$. %where $b_0\in(0,1)$ is a positive constant. 
	Then for any $(x,t)\in \mathbb{S}^n_+\times [0,T]$ $(T<T^*)$, we have
 	\begin{equation*}
 	|\nabla u|(x,t)\leq C,
 	\end{equation*}
 	where $C$ is a constant depends only on the initial values and the covariant derivatives with respect to round metric $\sigma$ on $\mathbb{S}^n_+$.
 \end{prop}

 \begin{proof}
 	Define a function
 	\begin{equation*}
 	\Phi:=(1+Kd)v+\cos\theta \sigma(\nabla u,\nabla d),
 	\end{equation*}
 	where $K>0$ is the positive constant to be determined later.
 	Assume that $\Phi$ attains its maximum value at $(x_0,t_0)\in\overline{ \mathbb{S}}^n_+\times [0,T]$. We divide it into the following three cases to complete the proof.
 	
	\
	
 	\noindent \textbf{Case 1}: $(x_0,t_0)\in\partial\mathbb{S}^n_+\times [0,T]$. At $x_0$, we choose local  coordinates such that $\frac{\partial}{\partial x_n}$ be the inner normal direction of $\partial\mathbb{S}^n_+$, which is exactly equal to $-\partial_{\beta}$ and corresponds to $\nabla d$. And let $\{x_{i}\}_{i=1}^{n-1}$ be the geodesic coordinate of $x_0\in\partial\mathbb{S}^n_+$. Along the geodesic $x_n=t$ ($0<t\leq \varepsilon$), one takes the parallel transport of tangential direction $\frac{\partial}{\partial x_{i}}$ $(1\leq i\leq n-1)$  to establish the geodesic coordinate in the neighborhood around point $x_0$ in $\overline{\mathbb{S}^n_+}$.
 %	we choose the  orthonormal frame $\{e_i\}_{i=1}^n$ such that  $e_n=-\partial_{\beta}$ is the inner normal vector field on $\partial\mathbb{S}^n_+$, 
 	
 	First, we notice that $\Phi=v+\cos\theta u_n$ on the boundary $\partial\mathbb{S}^n_+$ from the boundary condition in \eqref{MMCF with capillary u}. We denote $\nabla' u$ and $u_n$ 
	 the tangential and  the normal part of $\nabla u$ on the
 	boundary by our choice of coordinates above. The boundary condition $u_n=-\cos\theta v$ implies that
 	\begin{equation}\nonumber
 	\begin{split}
 	u_n^2=\cos^2\theta v^2=\cos^2 \theta (1+|\nabla' u|^2+u_n^2),
 	\end{split}
 	\end{equation}
 	in other words,
 	\begin{equation}\label{normal part control by tangential part on bdry}
 	\begin{split}
 	u_n^2=\cot^2 \theta (1+|\nabla' u|^2),
 	\end{split}
 	\end{equation}
 	Moreover  we have
 	\begin{equation}\nonumber
 	\begin{split}
 	\Phi&=v\sin^2\theta=\sqrt{1+|\nabla'u|^2+u_n^2} \sin^2\theta=|\csc\theta|\sqrt{1+|\nabla'u|^2}\sin^2\theta
 	\\&= \sqrt{1+|\nabla'u|^2}|\sin\theta|.
 	\end{split}
 	\end{equation}
 	From the Gauss-Weingarten equation we have
 	\begin{equation}\nonumber
 	\begin{split}
 	\nabla_nv&=\frac{  \nabla_ku\nabla_{nk}u}{v}
 	=\frac{\sum\limits_{i=1}^{n-1} \nabla_{i}u\nabla_{ni}u }{v}-\cos \theta \nabla_{nn}u
 	\\&=\frac{1}{v} \sum\limits_{i=1}^{n-1} \big(u_{i}u_{ni} +\sum_{j=1}^{n-1}u_{i}b_{ij }u_{j}\big)-\cos \theta \nabla_{nn}u
 	\\&=\sum\limits_{i=1}^{n-1}\frac{ u_{i}u_{ni}}{v} -\cos \theta \nabla_{nn}u,
 	\end{split}
 	\end{equation}
 	where $b_{ij}:=\sigma( \nabla _{e_{i}} e_n,e_{i})=0$ is the second fundamental form of $\partial \mathbb{S}^n_+$ in $\mathbb{S}^n_+$ for $1\leq i,j\leq n-1$.
 	Then at $x_0\in\partial\mathbb{S}^n_+$, from the Hopf lemma, it implies that
 	\begin{equation}\label{normal direction at boundary}
 	\begin{split}
 	0&\geq \nabla_n \Phi(x_0) 
 	=\nabla_nv+Kv\nabla_nd+\nabla_n(  u_kd_k)\cos \theta
 	\\&= \nabla_nv+Kv+  \nabla_{kn}ud_k\cos\theta+ u_k\nabla_{kn}d\cos\theta 
 	\\&=\frac{1}{v} \sum\limits_{i=1}^{n-1} u_{i}u_{ni}+Kv +  u_k\nabla_{kn}d\cos\theta .
 	\end{split}
 	\end{equation}
 	Since $\{ \partial_{x_i}\}_{i=1}^{n-1}$ are the tangential vector fields on $\partial\mathbb{S}^n_+$, for $1\leq i\leq n-1$, we have that	
 	\begin{equation}\nonumber
 	\begin{split}
 	0=\nabla'_{i}\Phi (x_0)&=v_{i}+u_{ni}\cos \theta.
 	\end{split}
 	\end{equation}
 	This implies that
 	\begin{equation}\label{tangential derivative to varphi}
 	\begin{split}
 	v_{i}=-u_{ni}\cos \theta.
 	\end{split}
 	\end{equation}
 	By differentiating the boundary condition of \eqref{MMCF with capillary u} and combining with \eqref{tangential derivative to varphi}  we have   that
 	\begin{equation}\nonumber
 	\begin{split}
 	u_{ni}&=-\nabla'_{i} (\cos\theta v)=\cos^2\theta u_{ni}.
 	\end{split}
 	\end{equation}
 	Since $|\cos\theta|<1$, we get
 	 \begin{equation}\label{term 1}
 	u_{ni}=0, \qquad \forall 1\leq i\leq n-1.
 	\end{equation}
 	Substituting equation \eqref{term 1} into equation \eqref{normal direction at boundary}, we conclude that
 	\begin{equation}\nonumber
 	\begin{split}
 	0& \geq \frac{1}{v} \sum\limits_{i=1}^{n-1} u_{i}u_{ni}+ K{v}  + u_k\nabla_{kn}d\cos\theta
 	=K{v} +u_k\nabla_{kn}d\cos\theta
 	\\& \geq \Phi\big( K\frac{1}{\sin^2\theta}-C_1\big),
 	\end{split}
 	\end{equation}
 	for some universial positive constant $C_1$. By choosing $K$ large enough, say $K:={2C_1}$, we get a contradiction. So \textbf{Case 1} is impossible.
 	
	\
	
 	\noindent\textbf{Case 2}: $(x_0,t_0)\in \overline{ \mathbb{S}^n_+}\times\{0\}$. In this case  we have
 	\begin{equation}\nonumber
 	\begin{split}
 	\Phi(x,t) &\leq \Phi(x_0,0)=(1+Kd)\sqrt{1+|\nabla u_0|^2}+\sigma( \nabla u_0,\nabla d)\cos\theta
 	\\& \leq C.
 	\end{split}
 	\end{equation}
 	%Since $|\cos\theta|\leq b_0<1$, 
	It yields  that
 	\begin{equation}
 	\sup_{\mathbb{S}^n_+\times[0,T]} v\leq C,
 	\end{equation}
 	where $C$ is a positive constant depending only on $n$ and $u_0$.	
 	
	\
	
 	\noindent\textbf{Case 3}: $(x_0,t_0)\in \mathbb{S}^n_+\times(0,T]$. In this case, we have
 	\begin{equation}\label{first derivative interior}
 	0=\nabla_i \Phi(x_0,t_0)=(1+Kd)v_i+Kd_iv+\cos\theta (u_ld_l)_i,
 	\end{equation} 	
 	for all $1\leq i\leq n$, or  equivalently, 
 	\begin{equation}\label{u_li}
 	\big((1+Kd)\frac{\nabla_lu}{v}+\nabla_ld\cos\theta)\nabla_{il}u=- \nabla_lu\nabla_{il}d\cos\theta-Kd_iv.
 	\end{equation}
 	At $(x_0,t_0)$, by rotating the geodesic coordinate $\{(x_1,x_2,\ldots,x_n)\}$  we may assume 
 	\begin{equation}\nonumber
 	|\nabla u|=u_1>0,\quad\text{and}\quad \{\nabla_{\alpha\beta}u\}_{2\leq \alpha,\beta\leq n}\quad \text{is diagonal}.
 	\end{equation}
 	We may also assume that $u_1(x_0,t_0)$ large enough in the below computation, such that $u_1, v=\sqrt{1+u_1^2}$, and $\Phi=(1+Kd)v+u_1d_1\cos\theta$ %(since we assume $|\cos\theta|\leq b_0<1$)
	 are equivalent to each other at $(x_0,t_0)$. Otherwise, we have completed the proof. All the computation below are done at the point $(x_0,t_0)$.
 	
 	First it is easy to see
 	\begin{equation}\label{u1 alpha u alpha}
 	\big[(1+Kd)\frac{u_1}{v}+\cos\theta d_1\big]u_{1\alpha}=-\cos\theta u_{\alpha\alpha}d_{\alpha}-\cos\theta u_1d_{1\alpha}-Kd_{\alpha}v,
 	\end{equation}
 	and
 	\begin{equation}\label{u 11 first}
 	\big[(1+Kd)\frac{u_1}{v}+\cos\theta d_1\big]u_{11}=-\cos\theta u_{\alpha 1}d_{\alpha}-\cos\theta u_1d_{11}-Kd_{1}v.
 	\end{equation}
 	Denote $S:=(1+Kd)\frac{u_1}{v}+\cos\theta d_1$.  It is easy to check   that $2+K\geq S\geq C(\delta,\theta)>0$ if we assume that $u_1\geq \delta>0$, otherwise we have obtained the estimate. Equation \eqref{u1 alpha u alpha} yields that
 	\begin{equation}\label{u1 alpha}
 	u_{1\alpha}=-\frac{\cos\theta d_{\alpha}}{S}u_{\alpha\alpha}-\frac{1}{S}\big(\cos\theta u_1d_{1\alpha}+Kd_{\alpha}v\big).
 	\end{equation}
 	Substituting equation \eqref{u1 alpha} into equation \eqref{u 11 first}, we conclude that
 	\begin{equation}\label{u_11}
 	\begin{split}
 	u_{11}&=-\frac{1}{S}\cos\theta u_{\alpha 1}d_{\alpha}+\frac{1}{S}\big(-\cos\theta u_1d_{11}-Kd_{1}v\big)
 	\\&=\frac{\cos^2\theta }{S^2}\sum_{\alpha=2}^{n}d_{\alpha}^2u_{\alpha\alpha}+\Big[ \sum_{\alpha=2}^{n}\frac{\cos\theta d_{\alpha}}{S^2}\big(\cos\theta u_1d_{1\alpha}+Kd_{\alpha}v \big)\\&\quad- \frac{1}{S}\big(\cos\theta u_1d_{11}+Kd_{1}v\big) \Big]
 	\\&=\frac{\cos^2\theta }{S^2}\sum_{\alpha=2}^{n}d_{\alpha}^2u_{\alpha\alpha}+O(v).
 	\end{split}
 	\end{equation}
 	On the other hand, we have
 	\begin{equation}\label{second derivative}
 	\begin{split}
 	0&\leq (\partial_t- F^{ij}\nabla_{ij}- F_{p_i}\nabla_i)\Phi
 	\\&
 	=\frac{(1+Kd)}{v} 
 	{u_l(u_{lt}- F^{ij}u_{lij}- F_{p_i}u_{li})}+d_k\cos\theta \big(u_{kt}-  F^{ij}u_{kij} - F_{p_i}u_{ki}\big)\\&\quad+ (1+Kd)\Big( \frac{F^{ij}u_lu_{li}u_ku_{kj}}{v^3}-  \frac{F^{ij}u_{li}u_{lj}}{v}\Big)	
 	-\big(2 F^{ij}u_{ki}d_{kj}\cos\theta	\\&\quad+2K F^{ij}d_iv_j\big) -\big( F^{ij}u_k d_{kij}\cos\theta 
 	+K F^{ij}d_{ij}v\big)
 	- F_{p_i}\big(Kd_iv+\cos\theta  u_kd_{ki}\big)
 	\\&:=
 	\text{J}_1+\text{J}_2+\text{J}_3+\text{J}_4+\text{J}_5+\text{J}_6.
 	\end{split}
 	\end{equation}
 	Next we  carefully handle these six terms one by one.  Differentiating the main equation in \eqref{MMCF with capillary u}, we get
 	\begin{align}\label{first derivative to equation}
 	u_{tk}= F^{ij}u_{ijk}+ F_{p_i}u_{ik}+F_{\rho}\rho u_k+F_{\beta}\sigma(\partial_{\beta},e_k).
 	\end{align}
 	Combining with the communicative formula on $\mathbb{S}^n_+$ 
\begin{align}
u_{ijk}=u_{kij}+u_j\sigma_{ik}-u_k\sigma_{ij},
\end{align}
 	we have
 	\begin{equation*}
 	\begin{split}
 	\text{J}_1&:=\frac{(1+Kd)}{v}  
 	{u_l(u_{lt}-F^{ij}u_{lij}-F_{p_i}u_{li})}
 %	\\&=\frac{(1+Kd)u_l}{v}\big(F_\rho \rho u_l+F_{\beta}\sigma( \partial_{\beta},e_l)\big)+\frac{(1+Kd)u_l}{v} \big(F^{il}u_i-u_l\mathcal{F}\big)
% 	\\&=\frac{(1+Kd)|\nabla u|^2}{2v^2} \rho(1-\frac{1}{\rho^2})\big(\frac{u_{11}}{v^2}+\sum_{\alpha=2}^{n}u_{\alpha\alpha}\big)-\frac{(1+Kd)|\nabla u|^2}{v} \rho\cdot \frac{n}{2}\big(1+\frac{1}{\rho^2}\big)\big(\frac{|\nabla u|^2}{v} 	\\&\quad +\cos\theta\sin\beta u_{\beta}\big)+\frac{(1+Kd)u_{\beta}}{v} \Big[-\frac{1}{v}a^{ij}u_{ij}\sin\beta +\frac{n\cos\beta}{v}u_{\beta} 	\\&\quad-\cos\theta \frac{\rho^2-1}{2\rho} n\cos\beta u_{\beta} -\frac{n\cos\theta}{2\rho}(\rho^2+1)\sin\beta\Big]+\frac{(1-n)(1+Kd)|\nabla u|^2}{\rho e^w v^2} 
 	\\&=\Big\{\frac{(1+Kd)|\nabla u|^2}{2v^2} \rho(1-\frac{1}{\rho^2})\frac{u_{11}}{v^2}-\frac{(1+Kd)u_{\beta}}{v^2}\sin\beta \frac{u_{11}}{v^2}\Big\}	
 	\\&\quad+
 	\Big\{\frac{(1+Kd)}{2v^2} |\nabla u|^2 (\rho-\frac{1}{\rho}) \sum_{\alpha=2}^{n}u_{\alpha\alpha}\Big\}
 	- \Big\{\frac{(1+Kd)|\nabla u|^2}{v} \rho\cdot \frac{n}{2}\big(1+\frac{1}{\rho^2}\big)\big(\frac{|\nabla u|^2}{v}
 	\\&\quad	
 	+\cos\theta\sin\beta u_{\beta}\big) \Big\}
 	+\Big\{-\frac{(1+Kd)u_{\beta}}{v^2} \sin\beta \sum_{\alpha=2}^{n}u_{\alpha\alpha}+  \frac{(1+Kd)u_{\beta}}{v}\Big[\frac{n\cos\beta}{v}u_{\beta}
 	\\&\quad-\cos\theta \frac{\rho^2-1}{2\rho} n\cos\beta u_{\beta} -\frac{n\cos\theta}{2\rho}(\rho^2+1)\sin\beta\Big]
 	+\frac{(1-n)(1+Kd)|\nabla u|^2}{\rho e^w v^2} \Big\}
 	\\&:=\text{J}_{11}+\text{J}_{12}+\text{J}_{13}+\text{J}_{14}.
 	\end{split}
 	\end{equation*}
 	Now we tackle the above terms one by one. First, by using equation \eqref{u_11}, we obtain that
 	\begin{align*}
 	\begin{split}
 	\text{J}_{11}
 	%&:=\frac{(1+Kd)|\nabla u|^2}{2v^2} \rho(1-\frac{1}{\rho^2})\frac{u_{11}}{v^2}-\frac{(1+Kd)u_{\beta}}{v^2}\sin\beta \frac{u_{11}}{v^2}
 	%\\&
 	&=\Big[\frac{(1+Kd)|\nabla u|^2}{2v^2} \rho(1-\frac{1}{\rho^2})\frac{1}{v^2}-\frac{(1+Kd)u_{\beta}}{v^2}\sin\beta \frac{1}{v^2}
 	\Big]\cdot \Big[ \frac{\cos^2\theta }{S^2}\sum_{\alpha=2}^{n}d_{\alpha}^2u_{\alpha\alpha}+O(u_1)\Big]
 	\\&=O(\frac{1}{v^2}) \sum_{\alpha=2}^{n}|u_{\alpha\alpha}|+O(\frac{1}{v}).
 	\end{split}
 	\end{align*}
 	It is also not difficult to show that $\text{J}_{14}=O(\frac{1}{v}) \sum_{\alpha=2}^{n}|u_{\alpha\alpha}|+O(v).$ 
 	$J_{12}$  will be considered 
 	later, together with $J_{22}$ and $J_{32}$,  and $J_{13}$ with $J_{23}$. See below.
 	%\begin{equation*}
 	%\begin{split}
 	%\text{J}_{14}&:=-\frac{(1+Kd)u_{\beta}}{v^2} \sin\beta \sum_{\alpha=2}^{n}u_{\alpha\alpha}+\frac{(1+Kd)u_{\beta}}{v}\Big[\frac{n\cos\beta}{v}u_{\beta}
 	%-\cos\theta \frac{\rho^2-1}{2\rho} n\cos\beta u_{\beta} \\&\quad-\frac{n\cos\theta}{2\rho}(\rho^2+1)\sin\beta\Big]+\frac{(1-n)(1+Kd)|\nabla u|^2}{\rho e^w v^2} 
 	%\\&=O(\frac{1}{v}) \sum_{\alpha=2}^{n}|u_{\alpha\alpha}|+O(v).
 	%\end{split}
 	%\end{equation*}
 	For the term $\text{J}_3$, we have
 	\begin{equation*}
 	\begin{split}
 	\text{J}_3&:=(1+K_1d)\Big( \frac{F^{ij}u_lu_{li}u_ku_{kj}}{v^3}-  \frac{F^{ij}u_{li}u_{lj}}{v}\Big)\\&=\frac{(1+Kd)}{\rho e^w v}\big(-\frac{1}{v^5}u_{11}^2-\frac{2}{v^3}\sum_{\alpha=2}^{n}u_{1\alpha}^2\big) -(1-\varepsilon)\frac{(1+Kd)}{\rho e^w v^2}\sum_{\alpha=2}^{n}u_{\alpha\alpha}^2\\&\quad-\varepsilon \frac{(1+Kd)}{\rho e^w v^2}\sum_{\alpha=2}^{n}u_{\alpha\alpha}^2
 	\\&:=\text{J}_{31}+\text{J}_{32}+\text{J}_{33}.
 	\end{split}
 	\end{equation*}
 	From equation \eqref{first derivative to equation}, we deduce that
 	\begin{equation*}
 	\begin{split}
 	\text{J}_2&:=d_k\cos\theta \big(u_{kt}-F^{ij}u_{kij} -F_{p_i}u_{ki}\big) 
 	\\&=\cos\theta \big(F_{\rho}\rho d_ku_k+F_{\beta} d_{\beta}\big)-\cos\theta \sigma( \nabla u,\nabla d) \mathcal{F}+\cos\theta F^{ij}u_id_j
 	\\&=\Big[\cos\theta \sigma( \nabla u,\nabla d) (\rho-\frac{1}{\rho}) \frac{u_{11}}{2v^3}-\cos\theta \sin\beta d_{\beta} \frac{u_{11}}{v^3}\Big]+\frac{\cos\theta}{2v} \sigma( \nabla u,\nabla d) (\rho-\frac{1}{\rho})\sum_{\alpha=2}^{n}u_{\alpha\alpha}
 	\\&\quad -\frac{n}{2}(\rho+\frac{1}{\rho})\cos\theta \sigma( \nabla u,\nabla d) \big(\frac{|\nabla u|^2}{v}+\cos\theta \sin\beta u_{\beta}\big) +\Big\{-\frac{\cos\theta \sin\beta}{v}d_{\beta}\sum_{\alpha=2}^{n}u_{\alpha\alpha}\\&\quad-\frac{(n-1)u_1^2+n}{\rho v^3 e^w} \cos\theta\sigma( \nabla u,\nabla d)
 	+\cos\theta\frac{\sigma( \nabla u,\nabla d)}{\rho v^3 e^w}+\cos\theta d_{\beta}\Big[ \frac{n\cos\beta}{v}u_{\beta} \\&\quad -\cos\theta\frac{\rho^2-1}{2\rho}n\cos\beta  u_{\beta}-\frac{n\cos\theta}{2\rho} \big({\rho^2}+1\big)\sin\beta  \Big]+\frac{n}{2} \cos^2\theta (1-\frac{1}{\rho^2})\cos\beta \rho \sigma( \nabla u,\nabla d)\Big\}
 	\\&:=\text{J}_{21}+\text{J}_{22}+\text{J}_{23}+\text{J}_{24}.
 	\end{split}
 	\end{equation*}
 	%In an analogous way, we treat above four terms different.
 	For these terms,  we first notice that $\text{J}_{24}=O(\frac{1}{v}) \sum_{\alpha=2}^{n}|u_{\alpha\alpha}|+O(v).$
 	%\begin{equation*}
 	%\begin{split}
 	%\text{J}_{24}&:=-\frac{\cos\theta \sin\beta}{v}d_{\beta}\sum_{\alpha=2}^{n}u_{\alpha\alpha}-\frac{(n-1)u_1^2+n}{\rho v^3 e^w} \cos\theta \sigma( \nabla u,\nabla d)
 	%+\cos\theta\frac{\sigma( \nabla u,\nabla d)}{\rho v^3 e^w}\\&\quad+\cos\theta d_{\beta}\Big[ \frac{n\cos\beta}{v}u_{\beta}  -\cos\theta\frac{\rho^2-1}{2\rho}n\cos\beta  u_{\beta}-\frac{n\cos\theta}{2\rho} \big({\rho^2}+1\big)\sin\beta  \Big]\\&\quad+\frac{n}{2} \cos^2\theta (1-\frac{1}{\rho^2})\cos\beta \rho\cdot \sigma( \nabla u,\nabla d)
 	%\\&=O(\frac{1}{v}) \sum_{\alpha=2}^{n}|u_{\alpha\alpha}|+O(v).
 	%\end{split}
 	%\end{equation*}
 	Equation \eqref{u_11} implies $\text{J}_{21}=O(\frac{1}{v^2}) \sum_{\alpha=2}^{n}|u_{\alpha\alpha}|+O(\frac{1}{v}).$
 	%\begin{equation*}
 	%\begin{split}
 	%\text{J}_{21}&:=\cos\theta\sigma( \nabla u,\nabla d) (\rho-\frac{1}{\rho}) \frac{u_{11}}{2v^3}-\cos\theta \sin\beta d_{\beta} \frac{u_{11}}{v^3}
 	%\\&=\big[\cos\theta \sigma( \nabla u,\nabla d)(\rho-\frac{1}{\rho}) \frac{1}{2v^3}-\cos\theta \sin\beta d_{\beta} \frac{1}{v^3} \big]\cdot \big[ \frac{\cos^2\theta }{S^2}\sum_{\alpha=2}^{n}d_{\alpha}^2u_{\alpha\alpha}+O(u_1)\big]
 	%\\&=O(\frac{1}{v^2}) \sum_{\alpha=2}^{n}|u_{\alpha\alpha}|+O(\frac{1}{v}).
 	%%\end{split}
 	%\end{equation*}
 	Furthermore, we get by using the arithmetic-geometric inequality 
 	\begin{equation*}
 	\begin{split}
 	&\text{J}_{12}+\text{J}_{22}+\text{J}_{32}
 	\\&:=\frac{(1+Kd)}{2v^2} |\nabla u|^2 (\rho-\frac{1}{\rho}) \sum_{\alpha=2}^{n}u_{\alpha\alpha}+\cos\theta \frac{\sigma( \nabla u,\nabla d)}{2v}(\rho-\frac{1}{\rho})\sum_{\alpha=2}^{n}u_{\alpha\alpha}
 	\\&\quad-(1-\varepsilon)\frac{1+Kd}{\rho e^w v^2}\sum_{\alpha=2}^{n} u_{\alpha\alpha}^2
 	\\&=S(\rho-\frac{1}{\rho})\frac{u_1}{2v}\sum_{\alpha=2}^{n}u_{\alpha\alpha}-(1-\varepsilon)\frac{1+Kd}{\rho e^w v^2}\sum_{\alpha=2}^{n}u_{\alpha\alpha}^2
 	\\&\leq \frac{S}{1+Kd}\cdot  \frac{(n-1)S(\rho-\frac{1}{\rho})^2 u_1^2}{16(1-\varepsilon)} \rho e^w
 	\\&\leq \frac{(n-1)(1+|\cos\theta|)S }{16(1-\varepsilon)} (\rho-\frac{1}{\rho})^2\rho e^wu_1^2.
 	\end{split}
 	\end{equation*}
 	Before continuing,  we fix a constant $b_0\in (|\cos \theta|, \frac{3n+1}{5n-1})$, for $|\cos \theta|<\frac {3n+1}{5n-1}$.
	If 
 	\begin{equation}
 	\begin{split}
 	\frac{|\nabla u|^2}{v}+\cos\theta \sin\beta u_{\beta}< (1-b_0)u_1,
 	\end{split}
 	\end{equation}
 	%since we have $|\cos\theta|<b_0<1$, then it follows that 
	then
 	\begin{equation*}
 	0<b_0-|\cos\theta|<b_0-\sin\beta |\cos\theta| |u_\beta| u_1^{-1}<1-\frac{u_1}{v}, 	\end{equation*}
	which implies that $u_1$ is uniformly bounded.
 	%which is a contradiction. %Hence it holds that $v\leq C$. 
 	Therefore, we may assume that \begin{equation}
 	\begin{split}
 	\frac{|\nabla u|^2}{v}+\cos\theta \sin\beta u_{\beta}\geq (1-b_0)u_1.
 	\end{split}
 	\end{equation}
 	Now we conclude that this yields
 	\begin{equation*}
 	\begin{split}
 	&\text{J}_{13}+\text{J}_{23}\\&:=-\frac{n}{2}\frac{1+Kd}{v}|\nabla u|^2 (\rho+\frac{1}{\rho}) \big(\frac{|\nabla u|^2}{v}+\cos\theta \sin\beta u_{\beta}\big)-	\frac{n}{2}\cos\theta \sigma( \nabla u,\nabla d) \\&\quad\big(\frac{|\nabla u|^2}{v}+\cos\theta \sin\beta u_{\beta}\big)
 (\rho+\frac{1}{\rho})
 	\\&=-\frac{n}{2}u_1 \big(\frac{|\nabla u|^2}{v}+\cos\theta \sin\beta u_{\beta}\big) (\rho+\frac{1}{\rho}) S
 	\\&\leq -\frac{n}{2}(1-b_0) S u_1^2  (\rho+\frac{1}{\rho}).
 	\end{split}
 	\end{equation*}
 	Since  $|\cos\theta|\leq b_0<\frac{3n+1}{5n-1}$,  by choosing $\varepsilon=\frac{\varepsilon_0}{2}\in (0,1)$ with  $\varepsilon_0:=\frac{3n+1-b_0(5n-1)}{4n(1-b_0)}>0$, we have  that $(n-1)(1+b_0)-4(1-\varepsilon)(1-b_0)n<0$. Then we deduce
 	\begin{equation*}
 	\begin{split}
 	&\text{J}_{13}+\text{J}_{23}+\text{J}_{12}+\text{J}_{22}+\text{J}_{32}
 	\\&\leq 
 	-\frac{n}{2}(1-b_0) S u_1^2  (\rho+\frac{1}{\rho})+u_1^2\frac{(n-1)(1+b_0)S }{16(1-\varepsilon)} (\rho-\frac{1}{\rho})^2\rho e^w
 	\\&\leq u_1^2S\big[ \frac{(n-1)(1+b_0)}{16(1-\varepsilon)} (\rho-\frac{1}{\rho})^2\frac{2\rho}{\rho^2+1}-(1-b_0)\frac{n}{2}(\rho+\frac{1}{\rho})\big]
% 	\\&= \frac{u_1^2S}{8\rho(\rho^2+1)(1-\varepsilon)} \big[ (n-1)(1+b_0)(\rho^2-1)^2-4(1-\varepsilon)(1-b_0)n(\rho^2+1)^2\big]
 	\\&=\frac{u_1^2S}{8\rho(\rho^2+1)(1-\varepsilon)} \Big\{\big[(n-1)(1+b_0)-4(1-\varepsilon)(1-b_0)n\big](\rho^4+1)
 	\\&\quad -\big[2(n-1)(1-b_0)+8(1-\varepsilon)(1-b_0)n\big]\rho^2 \Big\}
 	\\&\leq -\alpha_0 u_1^2,
 	\end{split}
 	\end{equation*}
 	where $\alpha_0$ is a positive constant, which only depends on $n,b_0$ and $\|u\|{_{C^0}}$.
 	Using equations \eqref{u1 alpha} and \eqref{u_11} again, we have
 	\begin{equation*} \begin{array}{rcl}
 	\text{J}_4+\text{J}_6&=& \ds\vs -2F^{ij}u_{ki}d_{kj}\cos\theta-2KF^{ij}d_iv_j-F_{p_i}\big(Kd_iv+\cos\theta \sum_{k=1}^{n}u_kd_{ki}\big)\\
% 	\\&=-\frac{2}{\rho ve^w} (\sigma^{ij}-\frac{u^iu^j}{v^2}) u_{ki}d_{kj}\cos\theta-2K \frac{u_1}{\rho v^2 e^w} \big(\sigma^{ij}-\frac{u^iu^j}{v^2}\big) d_iu_{1j}
 %	\\&=-\frac{2\cos\theta}{\rho ve^w} \Big[(1+\frac{1}{v^2}) \sum_{\alpha=2}^{n}u_{1\alpha}d_{1\alpha} +\frac{1}{v^2}  u_{11}d_{11}+\sum_{\alpha=2}^{n}u_{\alpha\alpha}d_{\alpha\alpha}      \Big]-2K \frac{u_1}{\rho e^w v^4}d_1u_{11} \\&\quad -\frac{2Ku_1}{\rho v^2 e^w} \sum_{\alpha=2}^{n} d_{\alpha}u_{1\alpha}
 &=& O(\frac{1}{v}) \sum_{\alpha=2}^{n}|u_{\alpha\alpha}|+O(1).
 	\end{array}\end{equation*}
 	For term $J_5$, it is easy to see 
 	$
 	\text{J}_5:=-F^{ij}u_k d_{kij}\cos\theta 
 	-KF^{ij}d_{ij}v
=O(1).
 $
 	
 %	and also 	\begin{equation*} 	\begin{split} 	\text{J}_{62}&:=\frac{u_1^2}{\rho v^3 e^w} a^{kl}u_{kl} \cos\theta d_{11} 	+\frac{2\cos\theta u_1^2}{\rho v^3 e^w} a^{il}u_{1l}d_{1i} 	\\&=\frac{u_1^2}{\rho v^3 e^w}  \cos\theta d_{11}\big(\frac{u_{11}}{v^2}+\sum_{\alpha=2}^{n}u_{\alpha\alpha}\big) 	+\frac{2\cos\theta u_1^2}{\rho v^3 e^w} \big(\frac{u_{11}d_{11}}{v^2}+\sum_{\alpha=2}^{n}u_{1\alpha}d_{1\alpha}\big) 	\\&=\frac{u_1^2}{\rho e^w v^5}\cos\theta d_{11}\big(\frac{\cos^2\theta }{S^2}\sum_{\alpha=2}^{n}d_{\alpha}^2u_{\alpha\alpha}+O(u_1)\big)+\frac{u_1^2}{\rho e^w v^3}\cos\theta d_{11} \sum_{\alpha=2}^{n}u_{\alpha\alpha} 	\\&\quad 	+\frac{2\cos\theta u_1^2}{\rho v^5 e^w}d_{11}\big(\frac{\cos^2\theta }{S^2}\sum_{\alpha=2}^{n}d_{\alpha}^2u_{\alpha\alpha}+O(u_1)\big) 	-\frac{2\cos\theta u_1^2}{\rho v^3 e^w} \sum_{\alpha=2}^{n}d_{1\alpha}\big[ \frac{\cos\theta d_{\alpha}}{S}u_{\alpha\alpha} 	\\&\quad  	+\frac{1}{S}\big(\cos\theta u_1d_{1\alpha}+Kd_{\alpha}v\big)\big] 	\\&=O(\frac{1}{v}) \sum_{\alpha=2}^{n}|u_{\alpha\alpha}|+O(1). 	\end{split} 	\end{equation*}
 	By adding all above terms into equation \eqref{second derivative},  we have
 	\begin{equation*}
 	\begin{split}
 	0& %\leq 	\text{J}_1+\text{J}_2+\text{J}_3+\text{J}_4+\text{J}_5+\text{J}_6
 	\leq \frac{(1+Kd)}{\rho e^w v}\big(-\frac{1}{v^5}u_{11}^2-\frac{2}{v^3}\sum_{\alpha=2}^{n}u_{1\alpha}^2\big) -\varepsilon_0 \frac{(1+Kd)}{2\rho e^w v^2}\sum_{\alpha=2}^{n}u_{\alpha\alpha}^2-\alpha_0 u_1^2\\&\quad+O(\frac{1}{v}) \sum_{\alpha=2}^{n}|u_{\alpha\alpha}|+O(v)
 	\\&\leq  -\varepsilon_0 \frac{(1+Kd)}{2\rho e^w v^2}\sum_{\alpha=2}^{n}u_{\alpha\alpha}^2+\frac{C_2}{v} \sum_{\alpha=2}^{n}|u_{\alpha\alpha}|-\alpha_0 u_1^2+C_1v
 	\\&\leq -\alpha_0 u_1^2+C_1v+\frac{C_2^2 \rho e^w}{2\varepsilon_0 (1+Kd)}.
 	\end{split}
 	\end{equation*}
 	Hence we conclude  that
 	\begin{equation*}
 	u_1\leq C.
 	\end{equation*}
 	 We have completed the proof.	
 \end{proof}
 
 \begin{rem}
 We remark that the condition $ |\cos\theta|\leq  b_0<\frac{3n+1}{5n-1}$ was only used  in the estimate of term $\text{J}_{13}+\text{J}_{23}+\text{J}_{12}+\text{J}_{22}+\text{J}_{32}$. And the main dominating term is J$_{13}$, which ensures us to obtain the gradient estimate under this contact angle range.
 \end{rem}
 The higher order a priori estimates of $u$ follow from the uniform $C^0$ and $C^1$ estimates.  Denote  $j(p):= {\sigma(p,\partial_\beta\rangle} -\cos\theta \sqrt{1+|p|^2}$ for $p\in\mathbb{R}^n$.   It is easy to see  that
 \begin{align*}
 \sigma(j_{p}\big|_{p=\nabla u},\partial_\beta)= 1-\cos^2\theta >1-b_0^2>0,
 \end{align*} which means that we have a uniformly oblique boundary condition. To be more precise, from the classical parabolic theory for quasi-linear parabolic equations (See \cite{LSU} for instance), it follows that
 
 \begin{prop}\label{higher estimates}
 	If $u(\cdot,t)$ solves the initial boundary value problem \eqref{MMCF with capillary u} on interval $[0, T^*)$ for $T^*\in (0, \infty]$ with $|\cos\theta|<\frac{3n+1}{5n-1}$, then for any $0<T<T^*$, we have
 	\begin{equation*}
 	\|u(\cdot,t)\|_{C^k}\leq C,\quad 0\leq t\leq T,
 	\end{equation*}
 	where $C$ is a positive constant only depends on $k$, and the initial values and the covariant derivatives with respect to the round metric on $\mathbb{S}^n_+$. It follows, in particular, $T^*=\infty$.
 \end{prop}

 \begin{proof}[\bf{Proof of Theorem} \ref{Main thm}]We only need to show  that each subsequential limit is a spherical cap.
 	
 	As is shown in the proof of Proposition \ref{energy de}, integrating the equation \eqref{energy deri} over $t\in [0,+\infty)$ and combining with Proposition \ref{higher estimates}, we have that
 	\begin{align*}
 	\int_0^{\infty}\int_{\mathbb{S}^n_+} \sum_{i<j} |\kappa_i-\kappa_j|^2 d\mu(y) dt\leq C,	\end{align*}
 	where $\kappa_i(y,t)$ is the principal curvature of radial graph at $(y,t)\in \mathbb{S}^n_+\times [0,\infty)$.
 	Due to the uniform estimates from Proposition \ref{higher estimates}, one can show that
 	\begin{align*}
 	\lim\limits_{t\to \infty}|\kappa_i-\kappa_j|^2=0, \quad \forall 1\leq i,j\leq n.
 	\end{align*}
 	Therefore any convergent subsequence of $x(\cdot,t)$ must converge to a  spherical cap as $t\to +\infty$. Moreover the capillary boundary condition implies that this spherical caps intersects  with the sphere at a  contact angle $\theta$.
	Hence it should belong to the family 
	given  in Remark \ref{static model}. Hence we have completed the proof of our main theorem.
 \end{proof}

\begin{acknowledgements}
This work is supported  partly by SPP 2026 of DFG ``Geometry at infinity''. We would like to thank the referee for his or her critical reading and helpful suggestion.
\end{acknowledgements}




\end{document}